\documentclass[11pt]{article}

\usepackage{amscd}
\usepackage{amsmath}
\usepackage{amssymb}
\usepackage{amstext}
\usepackage{amsthm}
\usepackage{bbold}
\usepackage{bm}
\usepackage{booktabs}
\usepackage{color}
\usepackage{colonequals}
\usepackage{comment}

\usepackage{easybmat}
\usepackage{etex}
\usepackage{enumitem}
\usepackage{framed}
\usepackage{authblk}
\usepackage[dvips,letterpaper,margin=1in]{geometry}
\usepackage{graphicx}
\usepackage{listings}
\usepackage{longtable}
\usepackage{mathtools}
\usepackage{multirow}
\usepackage{rotating}
\usepackage{setspace}
\usepackage{siunitx}
\usepackage{tabu}
\usepackage{verbatim}

\definecolor{cblack}{rgb}{0,0,0}
\definecolor{cblue}{rgb}{0.121569,0.466667,0.705882}    
\definecolor{corange}{rgb}{1.000000,0.498039,0.054902}  
\definecolor{cgreen}{rgb}{0.172549,0.627451,0.172549}   
\definecolor{cred}{rgb}{0.839216,0.152941,0.156863}     
\definecolor{cpurple}{rgb}{0.580392,0.403922,0.741176}  
\definecolor{cbrown}{rgb}{0.549020,0.337255,0.294118}   
\definecolor{cpink}{rgb}{0.890196,0.466667,0.760784}
\definecolor{cgray}{rgb}{0.498039,0.498039,0.498039}
\definecolor{cgreen2}{rgb}{0.7372549019607844, 0.7411764705882353, 0.13333333333333333}

\usepackage{hyperref}
\hypersetup{
  linkcolor  = cblue,
  citecolor  = cgreen,
  urlcolor   = corange,
  colorlinks = true,
}

\newtheorem{theorem}{Theorem}[section]
\newtheorem{remark}[theorem]{Remark}
\newtheorem{assumption}[theorem]{Assumption}
\newtheorem{lemma}[theorem]{Lemma}

\newtheorem{definition}[theorem]{Definition}
\newtheorem{example}[theorem]{Example}
\newtheorem{proposition}[theorem]{Proposition}
\newtheorem{corollary}[theorem]{Corollary}

\theoremstyle{plain} 
\newcommand{\thistheoremname}{}
\newtheorem*{genericthm}{\thistheoremname}

\makeatletter
\def\moverlay{\mathpalette\mov@rlay}
\def\mov@rlay#1#2{\leavevmode\vtop{%
   \baselineskip\z@skip \lineskiplimit-\maxdimen
   \ialign{\hfil$\m@th#1##$\hfil\cr#2\crcr}}}
\newcommand{\charfusion}[3][\mathord]{
    #1{\ifx#1\mathop\vphantom{#2}\fi
        \mathpalette\mov@rlay{#2\cr#3}
      }
    \ifx#1\mathop\expandafter\displaylimits\fi}
\makeatother

\newcommand{\CC}{\mathbb{C}}

\newcommand{\EE}{\mathbb{E}}

\newcommand{\NN}{\mathbb{N}}
\newcommand{\PP}{\mathbb{P}}
\newcommand{\QQ}{\mathbb{Q}}
\newcommand{\RR}{\mathbb{R}}

\newcommand{\ZZ}{\mathbb{Z}}
\DeclareSymbolFont{bbold}{U}{bbold}{m}{n}
\DeclareSymbolFontAlphabet{\mathbbold}{bbold}
\newcommand{\One}{\mathbbold{1}}

\newcommand{\ba}{\bm a}
\newcommand{\bb}{\bm b}

\newcommand{\be}{\bm e}

\newcommand{\bv}{\bm v}

\newcommand{\bx}{\bm x}
\newcommand{\by}{\bm y}
\newcommand{\bz}{\bm z}

\newcommand{\bD}{\bm D}

\newcommand{\bF}{\bm F}

\newcommand{\bM}{\bm M}

\newcommand{\bP}{\bm P}
\newcommand{\bQ}{\bm Q}

\newcommand{\bT}{\bm T}

\newcommand{\bW}{\bm W}
\newcommand{\bX}{\bm X}
\newcommand{\bY}{\bm Y}

\newcommand{\one}{\bm{1}}

\newcommand{\sN}{\mathcal{N}}

\newcommand{\sX}{\mathcal{X}}

\DeclarePairedDelimiter\ceil{\lceil}{\rceil}

\DeclareSymbolFont{sfoperators}{OT1}{cmss}{m}{n}
\DeclareSymbolFontAlphabet{\mathsf}{sfoperators}
\makeatletter
\renewcommand{\operator@font}{\mathgroup\symsfoperators}
\makeatother

\DeclareMathOperator{\sym}{sym}

\DeclareMathOperator{\Tr}{Tr}

\DeclareMathOperator{\Sym}{Sym}

\DeclareMathOperator{\Unif}{Unif}

\DeclareMathOperator{\Var}{Var}
\DeclareMathOperator{\Cov}{Cov}

\newcommand{\Px}{\mathop{\mathbb{P}}}  
\newcommand{\Ex}{\mathop{\mathbb{E}}}  
\newcommand{\Varx}{\mathop{\mathsf{Var}}}  

\newcommand{\xto}{\xrightarrow}

\newcommand{\Ber}{\mathsf{Ber}}
\newcommand{\cdeg}{\mathsf{cdeg}}

\newcommand{\Adv}{\mathsf{Adv}}
\newcommand{\CAdv}{\mathsf{CAdv}}

\renewcommand{\epsilon}{\varepsilon}
\newcommand{\avg}{\mathsf{avg}}
\newcommand{\inj}{\mathsf{inj}}
\newcommand{\bmu}{\bar{\mu}}
\newcommand{\bbz}{\bar{\bm z}}
\newcommand{\Pear}{\mathsf{Pear}}

\newcommand{\Mult}{\mathsf{Mult}}
\newcommand{\tlarge}{\,\mathsf{large}}
\newcommand{\tsmall}{\,\mathsf{small}}

\newcommand\numberthis{\addtocounter{equation}{1}\tag{\theequation}}

\title{Low coordinate degree algorithms II: Categorical signals and generalized stochastic block models}
\date{December 30, 2024}

\usepackage{authblk}
\author{Dmitriy Kunisky\thanks{Email: \textit{kunisky@jhu.edu}. Partially supported by ONR Award N00014-20-1-2335, a Simons Investigator Award to Daniel Spielman, and NSF grants DMS-1712730 and DMS-1719545.}}
\affil{Department of Applied Mathematics \& Statistics, Johns Hopkins University}

\begin{document}

\maketitle

\thispagestyle{empty}

\begin{abstract}
    We study when low coordinate degree functions (LCDF)---linear combinations of functions depending on small subsets of entries of a vector---can test for the presence of categorical structure, including community structure and generalizations thereof, in high-dimensional data.
    This complements the first paper of this series, which studied the power of LCDF in testing for continuous structure like real-valued signals perturbed by additive noise.
    We apply the tools developed there to a general form of stochastic block model (SBM), where a population is assigned random labels and every $p$-tuple of the population generates an observation according to an arbitrary probability measure associated to the $p$ labels of its members.
    We show that the performance of LCDF admits a unified analysis for this class of models.
    As applications, we prove tight lower bounds against LCDF (and therefore also against low degree polynomials) for nearly arbitrary graph and regular hypergraph SBMs, always matching suitable generalizations of the Kesten-Stigum threshold.
    We also prove tight lower bounds for group synchronization and abelian group sumset problems under the ``truth-or-Haar'' noise model, and use our technical results to give an improved analysis of Gaussian multi-frequency group synchronization.
    In most of these models, for some parameter settings our lower bounds give new evidence for conjectural statistical-to-computational gaps.
    Finally, interpreting some of our findings, we propose a precise analogy between categorical and continuous signals: a general SBM as above behaves, in terms of the tradeoff between subexponential runtime cost of testing algorithms and the signal strength needed for a testing algorithm to succeed, like a spiked $p_*$-tensor model of a certain order $p_*$ that may be computed from the parameters of the SBM.
\end{abstract}

\clearpage

\tableofcontents

\pagestyle{empty}

\clearpage

\setcounter{page}{1}
\pagestyle{plain}

\section{Introduction}

The first paper \cite{Kunisky-2024-LCDA1} in this series proposed \emph{low coordinate degree function (LCDF) algorithms} (reviewed in Section~\ref{sec:lcdf} below) as a fruitful alternative to the now widely-studied class of \emph{low degree polynomial (LDP) algorithms} for hypothesis testing \cite{HS-2017-BayesianEstimation, BHKKMP-2019-PlantedClique, HKPRSS-2017-SOSSpectral, Hopkins-2018-Thesis, KWB-2022-LowDegreeNotes}.
That work found that LCDF algorithms, though a more powerful and expressive class than LDP algorithms, are actually often easier to analyze.
In particular, it gave a unified analysis of the performance of LCDF for detecting weak high-dimensional signals observed through entrywise independent noisy channels, including models such as spiked matrices and tensors under general noise models.
In those results, the notion of a \emph{weak} signal has a specific quantitative meaning: the signal takes continuous values, say in $\RR$, and has typically to be close to zero.
In that case, the difficulty of a testing problem is governed only by the way that the noisy channel corrupts infinitesimal signals, which is why it turns out that, in that setting, channels do not all behave differently but rather fall into large \emph{universality classes}: only a single scalar summary statistic of their behavior on small signals characterizes how difficult they are for LCDF algorithms.

Here we continue to explore LCDF algorithms and the greater level of generality at which they allow us to analyze such statistical questions, but we focus on the complementary setting of \emph{categorical} signals.
The motivation for doing this is another broad class of models of interest in the high-dimensional statistics literature, for which the above assumption about quantitatively weak signals is not even wrong but just nonsensical: often a signal of interest is not a continuous object at all, but rather a combinatorial one, like the assignment of the nodes of a random graph to latent communities that govern nodes' probabilities of being connected to one another or not.
Sometimes, as discussed for certain stochastic block models in \cite{Kunisky-2024-LCDA1}, it is possible to shoehorn such models into the setting of continuous signals (after all, the community assignments above are not continuous, but the \emph{probabilities} of pairs being connected still are).
But this is not always possible or convenient, and we will see that developing a separate framework tailored to such combinatorial models leads to a different class of universality phenomena and intriguing and subtler parallels between testing for continuous and categorical signals.

\subsection{Generalized Stochastic Block Models}

We begin by describing the models that our results will treat.
These are a general form of \emph{stochastic block model (SBM)}, a model of random graphs originating in the social science literature; see \cite{Moore-2017-SBMReview, Abbe-2017-SBMReview} for recent theoretical surveys.
Our model is similar to the families of models considered in \cite{HLM-2012-LabelledSBM, LMX-2015-LabelledSBM} under the name of \emph{labeled SBMs}, and especially similar to the \emph{generalized SBM} of \cite{XML-2014-InferenceGeneralizedStochasticBlockModel}, though still differing slightly in the level of generality of different ingredients of the definition.
The definition is so similar in spirit to the latter, though, that we keep the same name.

\begin{definition}[Generalized stochastic block model]
    \label{def:gsbm}
    Let $p \geq 2$, let $k, n \geq 1$, and let $\Omega$ be a measurable space.
    A \emph{generalized stochastic block model (GSBM)} is specified by, for each $\ba \in [k]^p$, a probability measure $\mu_{\ba}$ on $\Omega$.
    Write
    \begin{equation}
        \mu_{\avg} \colonequals \frac{1}{k^p} \sum_{\ba \sim \Unif([k]^p)} \mu_{\ba},
        \label{eq:mu-avg}
    \end{equation}
    so that $\mu_{\avg}$ is another probability measure on $\Omega$.
    The GSBM then consists of the following two probability measures over $\bY \in \Omega^{\binom{[n]}{p}}$:
    \begin{enumerate}
    \item Under $\QQ$, draw $\bY \sim \QQ$ with $Y_S \sim \mu_{\avg}$ independently for each $S \in \binom{[n]}{p}$.
    \item Under $\PP$, first draw $\bx \sim \Unif([k]^n)$.
        Then, for each $S = \{s_1 < \cdots < s_p\} \in \binom{[n]}{p}$, draw $Y_S \sim \mu_{x_{s_1}, \dots, x_{s_p}}$ independently.
    \end{enumerate}
    We call $p$ the \emph{order} and $n$ the \emph{size} of a GSBM.
\end{definition}
\noindent
We will be interested in the problem of \emph{detection} or \emph{hypothesis testing} in such a model: given a draw from either $\QQ$ or $\PP$, can we with high probability determine which measure the observation was drawn from?
Precise details about our notion of testing follow below in Section~\ref{sec:lcdf}.

The class of GSBMs models essentially arbitrary situations where some collection of objects are each assigned discrete information denoted by their label in $[k]$, and where each subset of $p$ objects interacts in a way that depends only on their labels.\footnote{One further generalization that very likely could be treated by our methods but that we do not endeavor to handle for the sake of simplicity is where the assignment of labels is not uniformly random but is done according to some probability distribution on $[k]$.}
We then observe the outcome of all of these interactions, through a family of noisy channels described by the collection of $(\mu_{\ba})_{\ba \in [k]^p}$, which say how any collection of $p$ labels in $[k]$ gives rise to a random observation in $\Omega$.

In conventional stochastic block models on graphs or hypergraphs, what we have called ``labels'' above are assignments to communities, and $\Omega = \{0, 1\}$, so that outcomes of interactions may be interpreted as Boolean variables and the entire collection of outcomes as (the adjacency matrix or tensor of) a graph (if $p = 2$) or hypergraph (if $p \geq 3$).
We emphasize, though, that $\Omega$ need not equal $\{0, 1\}$ or even a discrete set at all, but can be absolutely arbitrary.
We handle but do not give applications where $\Omega$ is continuous, but one non-Boolean situation we will consider in some of our applications is when $\Omega$ is a finite group, for instance.

We conclude the discussion of the setting with the assumptions we make on a GSBM in our main results.
\begin{assumption}
    \label{assum}
    We always make the following assumptions on a GSBM:
    \begin{enumerate}
    \item The GSBM is \emph{non-trivial}: there exists $\ba \in [k]^p$ such that $\mu_{\ba} \neq \mu_{\avg}$.
    \item The GSBM is \emph{regular}: for all $\ba \in [k]^p$, the likelihood ratio $d\mu_{\ba} / d\mu_{\avg}$ belongs to $L^2(\mu_{\avg})$.\footnote{That the likelihood ratio exists, i.e., that $\mu_{\ba}$ is absolutely continuous to $\mu_{\avg}$, follows immediately from the definition \eqref{eq:mu-avg} of $\mu_{\avg}$.}
    \item The GSBM is \emph{weakly symmetric}: for all $\ba, \bb \in [k]^p$ and all permutations $\sigma \in \Sym([p])$,
        \[ \Ex_{y \sim \mu_{\avg}}\left[\frac{d\mu_{(a_1, \dots, a_p)}}{d\mu_{\avg}}(y) \cdot \frac{d\mu_{(b_1, \dots, b_p)}}{d\mu_{\avg}}(y)\right] = \Ex_{y \sim \mu_{\avg}}\left[\frac{d\mu_{(a_{\sigma(1)}, \dots, a_{\sigma(p)})}}{d\mu_{\avg}}(y) \cdot \frac{d\mu_{(b_{\sigma(1)}, \dots, b_{\sigma(p))}}}{d\mu_{\avg}}(y)\right]. \]
        Note that these expectations are well-defined and finite by the regularity assumption.
    \end{enumerate}
\end{assumption}
\noindent
Non-triviality excludes the case where $\PP = \QQ$, i.e.\ the two models we propose testing between are exactly the same probability measure, in which case of course testing is trivially impossible.
Regularity ensures that various objects akin to $\chi^2$ divergences that we will encounter are always finite.
Weak symmetry is slightly subtler.
One useful remark is that it is implied by the following stronger and more natural condition.
\begin{definition}
    We say that a GSBM is \emph{strongly symmetric} if, for all $\ba \in [k]^p$ and $\sigma \in \Sym([p])$ a permutation,
    \[ \mu_{(a_1, \dots, a_p)} = \mu_{(a_{\sigma(1)}, \dots, a_{\sigma(p)})}. \]
\end{definition}
\begin{proposition}
    A strongly symmetric GSBM is also weakly symmetric.
\end{proposition}
\noindent
But, weak symmetry suffices for our calculations, and we will encounter at least one interesting model, namely synchronization over finite groups, where only weak, but not strong, symmetry is satisfied (in short, because there pairs of group elements interact such that from the pair $g, h$ one receives a noisy version of the asymmetric group difference $gh^{-1}$).
One may think of strong symmetry as asking that the observation channel $\mu_{\ba}$ truly has no ``direction'' it imposes among its inputs $(a_1, \dots, a_p)$, while weak symmetry allows for such a directionality but only asks that a certain measure of the ``relative similarity'' of distributions of observations through the channel does not depend on it.

\subsection{Low Coordinate Degree Algorithms}
\label{sec:lcdf}

We will be interested in the following specific notion of success in hypothesis testing.\footnote{We take a somewhat different perspective here than in \cite{Kunisky-2024-LCDA1}, and believe that this viewpoint is a sounder one than the more heuristic discussion of the meaning of the low coordinate degree advantage in that work.}
\begin{definition}[Strong separation]
    Consider a sequence of pairs of probability measures $\PP_n, \QQ_n$ over measurable spaces $\Omega_n$.
    We say that functions $f_n: \Omega_n \to \RR$ achieve \emph{strong separation} if
    \begin{equation}
        \Ex_{\bY \sim \PP_n} f_n(\bY) - \Ex_{\bY \sim \QQ_n} f_n(\bY) = \omega\left(\sqrt{\Varx_{\bY \sim \QQ_n} f_n(\bY)} + \sqrt{\Varx_{\bY \sim \PP_n} f_n(\bY)}\right)
    \end{equation}
    as $n \to \infty$.
\end{definition}
\noindent
Strong separation is a version of it being possible to hypothesis test consistently: given a strongly separating family of functions, one may choose a suitable threshold and obtain a sequence of hypothesis tests having Type~I and~II errors both $o(1)$ as $n \to \infty$, by Chebyshev's inequality.
In fact, a simple argument shows that the converse is also true, and strong separation by \emph{some} family of functions is equivalent to the above notion of consistent testing, often called \emph{strong detection}.

However, we will ask not whether there exist \emph{any} functions achieving strong separation, but whether there exist \emph{computationally tractable} such functions.
As alluded to above, much prior work has focused on this question for low degree polynomial (LDP) functions.
It should already be clear, however, that this will not work for the models we have proposed to study, since our observations will take values in arbitrary spaces.
Instead, and as explored in \cite{Kunisky-2024-LCDA1} based on an early proposal by Hopkins \cite{Hopkins-2018-Thesis}, we focus on the following more general class of functions that also have the benefit of being well-defined much more generally.\footnote{These functions also played an important role in the recent work \cite{BBHLS-2020-SQLowDegree}.}

Let $\Omega$ be a measurable space and $\QQ$ be a product measure on $\Omega^N$ for some $N \geq 1$, as in Definition~\ref{def:gsbm}.
We follow the definitions from \cite{Kunisky-2024-LCDA1}: for $\by \in \Omega^N$ and $T \subseteq [N]$, we write $\by_T \in \Omega^T$ for the restriction of $\by$ to the coordinates in $T$.
We define subspaces of $L^2(\QQ)$,
\begin{align}
  V_T &\colonequals \{f \in L^2(\QQ): f(\by) \text{ depends only on } \by_T\}, \\
        V_{\leq D} &\colonequals \sum_{\substack{T \subseteq [N] \\ |T| \leq D}} V_T.
\end{align}
\begin{definition}[Coordinate degree]
    For $f \in L^2(\QQ)$, we define $\cdeg(f) \colonequals \min\{D: f \in V_{\leq D}\}$, and call this the \emph{coordinate degree} of $f$.
\end{definition}
\noindent
We call $V_{\leq D}$ a space of low coordinate degree functions (LCDF), and this is the precise meaning of this term mentioned earlier.

Functions of coordinate degree at most $D$, when $\Omega \subseteq \RR$, include polynomials of degree at most $D$ (as each monomial of degree $D$ depends on at most $D$ coordinates).
Thus, when we prove lower bounds against LCDF we also always learn lower bounds against LDP, in the style of the results cited above as well as recent works like \cite{BKW-2019-ConstrainedPCA, DKWB-2019-SubexponentialTimeSparsePCA,BBKMW-2020-SpectralPlantingColoring,BAHSWZ-2022-FranzParisiLowDegree, BBHLS-2020-SQLowDegree,RSWY-2022-CountCommunitiesLowDegree,KVWX-2023-LowDegreeColoringClique}.
But, LCDF also~(1)~are sensible for arbitrary $\Omega$, for instance $\Omega$ an object like a finite group whose elements have only algebraic meaning, and (2) include many more functions, such as LDP applied after arbitrary entrywise functions taking numerical values (say, low-degree polynomials of a representation applied to a vector of group elements).\footnote{It is worth noting, however, that for finite sets $\Omega$ coordinate degree is the same as polynomial degree when $\Omega$ is given the ``one-hot'' Boolean encoding, as discussed in \cite{KM-2021-ReconstructionTreesLowDegree,Kunisky-2024-LCDA1}.}

We now have all of the ingredients needed to describe the template that all of our main results will follow: we will show, for sequences of hypothesis testing problems of $\QQ_n$ versus $\PP_n$ arising from sequences of GSBMs for growing $n$, that sequences of LCDF $(f_n)_{n \geq 1}$ with  $\cdeg(f_n) \leq D(n)$ cannot strongly separate $\QQ_n$ and $\PP_n$.
In particular, as for LDP (see, e.g., the discussion in \cite{KWB-2022-LowDegreeNotes}), we may think intuitively of LCDF of coordinate degree $D(n)$ as describing all functions taking roughly time $n^{D(n)} = \exp(\widetilde{\Theta}(D(n)))$ to compute.
Thus, these kinds of lower bounds suggest, depending on the scaling of $D(n)$, that various orders of superpolynomial but subexponential time are required to solve testing problems.

\subsection{Main Results: General Theory}

At the center of the theory we develop around the performence of LCDF on GSBMs is the following object, which, while perhaps mysterious for the moment, we will see arises naturally in calculations of likelihoods with GSBMs, and, more to the point, contains all of the important information about the difficulty of testing in a GSBM.

\begin{definition}[Characteristic tensor]
    For a GSBM specified by $(\mu_{\ba})_{\ba \in [k]^p}$, we define its \emph{characteristic tensor} (or \emph{matrix} if $p = 2$) to be $\bT = \bT^{(p)} \in (\RR^{[k] \times [k]})^{\otimes p}$ having entries
    \[ T_{(a_1, b_1), \dots, (a_p, b_p)} = \frac{1}{p!} \Ex_{y \sim \mu_{\avg}}\left[\left(\frac{d\mu_{(a_1, \dots, a_p)}}{d\mu_{\avg}}(y) - 1\right)\left(\frac{d\mu_{(b_1, \dots, b_p)}}{d\mu_{\avg}}(y) - 1\right)\right]. \]
\end{definition}
\noindent
To hint at the statistical relevance of $\bT$, note that the entries $T_{(a_1, a_1), \dots, (a_p, a_p)}$ are proportional to the $\chi^2$ divergence between $\mu_{\ba}$ and $\mu_{\avg}$, which is indeed a measurement of how atypical the observations under $\PP$ are compared to those under $\QQ$.
The remaining entries contain a kind of ``cross--$\chi^2$ divergence'' between different $\mu_{\ba}$ relative to $\mu_{\avg}$.
The reader interested in these matters is encouraged to look at the discussion in \cite{Kunisky-2024-LCDA1}; for the sake of brevity, we will not revisit those details here.

We note also that, for a weakly symmetric GSBM, $\bT$ as defined above is a symmetric tensor, but this would not be the case if we viewed $\bT$ instead as a tensor in $(\RR^k)^{\otimes 2p}$, not flattening $(a_i, b_i)$ into an element of $[k]^2$.

For the next definition and several to follow, a small amount of notation for working with tensors and their ``contractions'' will be useful.
\begin{definition}[Partial tensor contraction]
    Let $\bT \in (\RR^N)^{\otimes p}$ be a symmetric tensor and let $\bv_1, \dots, \bv_m \in \RR^N$ for some $1 \leq m \leq p$.
    We write $\bT[\bv_1, \dots, \bv_m, \cdot, \dots, \cdot\,] \in (\RR^N)^{\otimes p - m}$ for the tensor with entries
    \[ (\bT[\bv_1, \dots, \bv_m, \cdot\,, \dots, \cdot\,])_{i_1, \dots, i_{p - m}} = \sum_{j_1, \dots, j_m = 1}^N T_{j_1, \dots, j_m, i_1, \dots, i_{p - m}} (\bv_1)_{j_1} \cdots (\bv_m)_{j_m}. \]
\end{definition}
\noindent
The reader familiar with tensor network notation may view this as attaching $\bv_1, \dots, \bv_m$ to $m$ of the ``axes'' of $\bT$, while leaving the remaining $p - m$ axes free.

\begin{definition}[Marginal characteristic tensors]
    From the characteristic tensor $\bT^{(p)}$ of a GSBM as above, we further define a sequence of tensors $\bT^{(p - j)} \in (\RR^{[k] \times [k]})^{\otimes (p - j)}$ by
    \[ \bT^{(p - j)} = \frac{1}{k^{2j}} \bT^{(p)}[\,\underbrace{\one, \dots, \one}_{j \text{ times}}, \underbrace{\cdot\,, \dots, \cdot}_{p - j \text{ times}}\,], \]
    where $\one$ is the vector all of whose entries are 1 (in this case of dimension $k^2$).
\end{definition}
\noindent
The reason for the name is that, as one may check, $\bT^{(p - j)}$ is the characteristic tensor of another GSBM, now with $(p - j)$-way interactions, parametrized by $\mu_{\ba}^{(p - j)}$ formed by \emph{marginalizing} the $\mu_{\ba}$ defining the original GSBM, in the sense that
\[ \mu^{(p - j)}_{a_1, \dots, a_{p - j}} = \frac{1}{k^j}\sum_{a_{p - j + 1}, \dots, a_p = 1}^k \mu_{a_1, \dots, a_p}. \]
Alternatively, one samples from $\mu^{(p - j)}_{\widetilde{\ba}}$ by extending $\widetilde{\ba} \in [k]^{p - j}$ from a $(p - j)$-tuple to a $p$-tuple by adding $j$ entries from $[k]$ uniformly at random to form $\ba \in [k]^p$, and then sampling from the corresponding $\mu_{\ba}$.

The basic idea that our calculations will reflect is that, whenever one observes a GSBM, one observes many nearly-independent copies of its marginal GSBMs as well.
For instance, the ordinary graph SBM is a GSBM with $p = 2$ describing community structure in a random graph: each $\mu_{(a, b)}$ is a Bernoulli measure describing how much communities $a$ and $b$ tend to interact.
The case $j = 1$ of marginalization describes the degree distributions of the vertices in different communities: in the above interpretation, to sample from $\mu^{(1)}_{(a)}$, one chooses a random $b \in [k]$, i.e., a random community, and draws from $\mu_{(a, b)}$.
So, $\mu^{(1)}_{(a)}$ is another Bernoulli measure, now just describing how much community $a$ tends to connect to \emph{all} other communities.
And indeed, when we observe a random graph, we observe the connection patterns of all $n$ vertices to the other vertices, which look approximately like $n$ independent draws from the $j = 1$ marginalization.

More generally, the correct intuition is that from one observation of a GSBM we get information that looks approximately like $\Theta(n^j)$ draws of the marginalization to a model of order $p - j$.
This large number of draws will carry a large amount of signal for hypothesis testing, \emph{unless} $\bT^{(p - j)} = \bm 0$, in which case the marginalized model is what we have called \emph{trivial}---in the above SBM example, this happens if the average degrees of members of any particular community are the same, so just looking at degrees cannot be helpful for testing regardless of how many ``effective draws'' from this marginalized model we receive.
By making this precise, we will see that the difficulty of detection in a GSBM is governed by the greatest amount of marginalization we can perform before reaching a vanishing characteristic tensor, to which we give the following name:

\begin{definition}[Marginal order]
    The \emph{marginal order} of a non-trivial GSBM is the smallest $p_*$ for which the marginal characteristic tensor $\bT^{(p_*)} \neq \bm 0$ (equivalently, for which the marginal model of order $p_*$, in the above sense, is non-trivial).
\end{definition}

By construction, $\bT^{(0)}$ is the scalar zero, and thus we always have $p_* > 0$.
On the other hand if $\bT^{(p)} = \bm 0$ then we must have $\mu_{(a_1, \dots, a_p)} = \mu_{\avg}$ for all $\ba$, in which case $\PP = \QQ$ and the original model is trivial.
So, under our assumption that a GSBM is non-trivial, we always have $p_* \leq p$ (in particular $p_*$ is always well-defined, hence our assumption of non-triviality in the definition), and the range of possible values of the marginal order in this case is
\[ 1 \leq p_* \leq p. \]

We are now ready to state our main abstract results, which give simple conditions for lower bounds against LCDF for sequences of GSBMs in terms of the characteristic tensors of a model and its marginalizations.
\begin{definition}[Injective norm]
    For a symmetric tensor $\bX \in (\RR^N)^{\otimes p}$, its \emph{injective norm} is
    \[ \|\bX\|_{\inj} \colonequals \max_{\substack{\bv \in \RR^N \\ \|\bv\| = 1}} |\langle \bX, \bv^{\otimes p} \rangle| = \max_{\substack{\bv \in \RR^N \\ \|\bv\| = 1}} \left|\sum_{i_1, \dots, i_p = 1}^d X_{i_1, \dots, i_d} v_{i_1} \cdots v_{i_d}\right|. \]
\end{definition}

\begin{remark}
    It seems likely from examining our technical calculations that the ``right'' quantity that should appear in these results is not the injective norm but the tensor analog of the maximum eigenvalue,
    \[ \max_{\substack{\bv \in \RR^N \\ \|\bv\| = 1}} \langle \bX, \bv^{\otimes p} \rangle. \]
    However, there appear to be technical obstructions to achieving this, and we have not found any examples where our current methods give loose estimates while a version improved in this way would be significantly sharper.
\end{remark}

The shared setup for the two results below will be as follows.
Let $p \geq 2$, $k \geq 1$, and $\Omega$ a measurable space.
Consider a sequence of non-trivial GSBMs, with $p, k, \Omega$ fixed but with size $n \geq 1$ increasing, and $\mu_{\ba} = \mu_{n, \ba}$ possibly depending on $n$, and suppose these lead to characteristic tensors and corresponding marginalizations $\bT^{(j)}_n$.
Suppose every GSBM in this sequence, for sufficiently large $n$, has marginal order at least $p_* \in [p]$.\footnote{Results that we state for lower marginal order always also apply to higher marginal order---higher marginal order is a stronger assumption, asking that more marginalized characteristic tensors be zero.}
Let $\QQ = \QQ_n$ and $\PP = \PP_n$ be the sequence of probability measures over $\Omega^{\binom{[n]}{p}}$ corresponding to these GSBMs.

\begin{theorem}[General lower bound for general marginal order]
    \label{thm:lcdf-p3}
    Suppose that $p_* \geq 2$ in the above setting.
    There is a constant $c = c_{p, k}$ depending only on $p$ and $k$ such that, if for all sufficiently large $n$ we have $D(n) \leq cn$ and
    \begin{equation}
        \max_{p_* \leq j \leq p} \|\bT^{(j)}_n\|_{\inj} \leq c n^{-\frac{2p - p_*}{2}}D(n)^{-\frac{p_* - 2}{2}},
        \label{eq:lcdf-p3}
    \end{equation}
    then no sequence of functions of coordinate degree at most $D(n)$ can strongly separate $\QQ_n$ from $\PP_n$.
\end{theorem}

\begin{theorem}[General lower bound for marginal order 2]
    \label{thm:lcdf-p2}
    Suppose that $p_* = 2$ in the above setting.
    Note that in this case $\bT^{(p_*)} = \bT^{(2)}$ is a symmetric matrix and $\|\bT^{(2)}\|_{\inj} = \|\bT^{(2)}\|$ is equivalently the operator norm.
    Suppose that, for constants $C > 0$ and $\epsilon \in (0, 1)$, for all sufficiently large $n$,
    \begin{align*}
      \|\bT^{(2)}_n\| &\leq (1 - \epsilon) \frac{k^2}{p(p - 1)} \frac{1}{n^{p - 1}}, \\
      \|\bT^{(j)}_n\|_{\inj} &\leq C\frac{1}{n^{p - 1}} \text{ for all } 3 \leq j \leq p.
    \end{align*}
    Then, no sequence of functions of coordinate degree at most $D(n) = O(n \, / \log n)$ can strongly separate $\QQ_n$ from $\PP_n$.
\end{theorem}

A few comments are in order.
First, at a high level, we may view these as somewhat akin to universality results like those studied in \cite{Kunisky-2024-LCDA1} for testing for continuous signals: the results state that, while the entire collection of channel measures $(\mu_{\ba})$ is potentially a complicated object, our lower bounds depend only on the finite-dimensional characteristic tensor $\bT$.
We will see below some examples of different models with the same characteristic tensor, for which our results are then identical.
Since these lower bounds are often tight (though our results do not imply that), it seems that categorical testing problems might have similar universality phenomena to continuous ones, though further investigation would be needed to establish that rigorously.

Next, the second result should be viewed as an elaboration of the first in the special case $p_* = 2$: in that case, the condition of the first result asks that $\|\bT^{(j)}_n\|_{\inj} \leq c / n^{p - 1}$ for an unspecified constant $c$.
The second result instead gives a condition involving a precise constant $\frac{k^2}{p(p - 1)}$ in the bound on $\|\bT^{(2)}_n\|$ (and clarifies that the other injective norms may be bounded much more crudely), which we will see often leads to tight analyses giving lower bounds precisely complementary to the best known algorithmic results.

Finally, when $p_* \geq 3$, then the condition on the injective norms in the first result depends on the choice of the degree $D(n)$.
Thus in this case our lower bound will include a smooth tradeoff between the amount of signal required for testing to succeed and the available computational budget as measured by $D(n)$.
As mentioned before, we may think of LCDF of coordinate degree $D(n)$ as taking roughly time $\exp(\widetilde{\Theta}(D(n)))$ to compute.
When we consider $D(n)$ polynomial in $n$, say $D(n) \sim n^{\delta}$ for small $\delta$, then, since the other term on the right-hand side of \eqref{eq:lcdf-p3} is also polynomial in $n$, our lower bound will allow for a regime of \emph{subexponential time} algorithms, where increasing $\delta$ allows for testing for substantially weaker signals (as measured by the norms of the characteristic tensors).
In contrast, together with prior algorithmic results, we will see that Theorem~\ref{thm:lcdf-p2} for $p_* = 2$ leads to much sharper threshold phenomena: outside of a very narrow window of parameters, either a polynomial-time algorithm can solve a testing problem, or coordinate degree close to $n$, i.e., nearly exponential time according to the previous heuristic, is required.

This state of affairs exactly mirrors the difference in behaviors between spiked matrix and spiked tensor models.
For testing between a $p$-ary tensor of i.i.d.\ Gaussian noise and one with a rank-one tensor $\lambda \bx^{\otimes p}$ added, there is precisely the same kind of tradeoff between degree and signal-to-noise ratio $\lambda$ as in Theorem~\ref{thm:lcdf-p3}.
Indeed, replacing the left-hand side of \eqref{eq:lcdf-p3} with $\lambda^2$ and setting $p = p_*$, one obtains exactly the statement of an optimal lower bound against LDP for the $p$-ary spiked tensor model!\footnote{For $\bx$ suitably normalized.}
The reader may compare Theorem~\ref{thm:lcdf-p3} with Theorem~3.3 of \cite{KWB-2022-LowDegreeNotes}.
In contrast, for $p = 2$, this becomes a spiked matrix model, which exhibits precisely the kind of sharper computational threshold described in our Theorem~\ref{thm:lcdf-p2}; cf.\ Theorem~3.9 of \cite{KWB-2022-LowDegreeNotes}.

While here we only probe this intriguing analogy from the point of view of \emph{lower} bounds, similar parallels have also been observed, at least in specific situations, algorithmically as well.
For instance, in Section~\ref{sec:xor} we will treat the example of detecting planted solutions in random XOR-SAT formulas viewed as a GSBM, and in this case \cite{WEAM-2019-KikuchiTensorPCA} devised a unified spectral algorithm for both this task and an analogous tensor PCA problem.
In any case, we hope that these observations will bring attention to what seems to be a valid high-level principle concerning these models:
\[ \text{``A GSBM of marginal order } p_* \text{ behaves qualitatively like a spiked $p_*$-tensor model.''} \]

\subsection{Main Results: Applications}

\subsubsection{Graph Stochastic Block Models}
We begin with applications to the familiar graph-valued SBMs discussed above.
The following is a general form of such a model, with some special assumptions particular to our toolkit.

\begin{definition}[Stochastic block model]
    \label{def:sbm}
    Suppose $k \geq 2$ and $\bQ \in [0, 1]^{k \times k}_{\sym}$ is a symmetric matrix satisfying
    \[ \bQ\one = \lambda \one, \]
    so that $\lambda = \lambda_1(\bQ)$ is the largest (Perron-Frobenius) eigenvalue of $\bQ$.
    Define
    \[ q \colonequals \frac{1}{k^2}\sum_{a, b = 1}^k Q_{ab} = \frac{1}{k^2}\one^{\top}\bQ\one = \frac{\lambda}{k}. \]
    Associated to such a matrix, we define the \emph{stochastic block model (SBM)} consisting, for each $n \geq 1$, of two probability measures over graphs on $n$ vertices or, equivalently, their adjacency matrices, $\bY \in \{0, 1\}^{n \times n}_{\sym}$:
    \begin{enumerate}
    \item Under $\QQ_n$, draw $Y_{ij} \sim \Ber(\frac{1}{n} \cdot q)$ independently for each $1 \leq i < j \leq n$.
    \item Under $\PP_n$, first draw $\bx \in [k]^n$ uniformly at random.
        Then, draw $Y_{ij} \sim \Ber(\frac{1}{n} \cdot Q_{x_ix_j})$ independently for each $1 \leq i < j \leq n$.
    \end{enumerate}
    We call $\bQ$ the \emph{interaction matrix} of such a model.
\end{definition}
\noindent
The condition that $\one$ is an eigenvector of $\bQ$ will ensure, in the language we have introduced above, that this model has marginal order at least 2.

\begin{theorem}[Lower bound for stochastic block model]
    \label{thm:sbm}
    Suppose that $\bQ$ is an interaction matrix as in Definition~\ref{def:sbm}.
    If
    \begin{equation}
        \max_{j \in \{2, \dots, k\}} |\lambda_j(\bQ)|^2 < k \lambda_1(\bQ), \label{eq:ks}
    \end{equation}
    then no sequence of functions of coordinate degree $O(n\, / \log n)$ can strongly separate $\QQ_n$ from $\PP_n$ in the SBM associated to $\bQ$.
\end{theorem}
\noindent
The threshold \eqref{eq:ks} is a general form of the \emph{Kesten-Stigum (KS) threshold} for SBMs studied extensively in previous literature and shown to be the threshold where numerous algorithms cease to succeed at testing \cite{DKMZ-2011-AsymptoticAnalysisSBM, DKMZ-2011-SBM, BLM-2015-NonBacktrackingSpectrumSBM,AS-2018-SBMAchievability}; see also extensive discussion in \cite{Abbe-2017-SBMReview}.
Theorem~\ref{thm:sbm} thus gives a complementary lower bound to all of these results about concrete algorithms.
It is also known (and discussed in the above references) that, once $k \geq 4$, there are inefficient algorithms that can test below this threshold.
In these cases, Theorem~\ref{thm:sbm} also gives evidence for a \emph{statistical-to-computational gap}, a parameter regime where the testing problem is possible to solve but only by a very slow computation.

\begin{example}
    Consider the \emph{symmetric SBM}, which corresponds to the choice
    \[ Q_{ab} = \left\{ \begin{array}{ll} \alpha & \text{if } a = b, \\ \beta & \text{if } a \neq b \end{array}\right\}, \]
    for some $\alpha, \beta \geq 0$.
    Straightforward calculations show that \eqref{eq:ks} then reduces to the condition
    \[ \frac{(\alpha - \beta)^2}{k(\alpha + (k - 1)\beta)} < 1, \]
    perhaps a more familiar special case of the KS threshold.
    In this case, the result of Theorem~\ref{thm:sbm} was obtained by \cite{BBKMW-2020-SpectralPlantingColoring}, but the general case (to the best of our knowledge) has not been treated before.
    We derive the even more special subcase $k = 2$ of the symmetric SBM from our tools as a warmup in Section~\ref{sec:sbm-warmup}.
\end{example}

\subsubsection{Hypergraph Stochastic Block Models}
We next consider a version of the above results for SBMs defined over hypergraphs rather than graphs.
These results will actually strictly generalize the preceding ones, but we state and prove them separately for the sake of exposition to the reader more familiar with the graph SBM literature.

\begin{definition}[Hypergraph stochastic block model]
    \label{def:hsbm}
    Suppose $k \geq 2$, $p \geq 3$, and $\bQ \in ([0, 1]^{k})^{\otimes p}$ is a symmetric tensor satisfying
    \[ \bQ[\one, \dots, \one, \cdot\,] = \lambda\one. \]
    Define
    \[ q \colonequals \frac{1}{k^p}\sum_{a_1, \dots, a_p = 1}^k Q_{a_1 \cdots a_p} = \frac{1}{k^p}\langle \bQ, \one^{\otimes p} \rangle = \frac{\lambda}{k^{p - 1}}. \]
    Associated to such a tensor, we define the \emph{hypergraph stochastic block model (HSBM)} consisting, for each $n \geq 1$, of two probability measures over $p$-regular hypergraphs on $n$ vertices or, equivalently, their symmetric adjacency tensors, $\bY \in (\{0, 1\}^{n})^{\otimes p}$:
    \begin{enumerate}
    \item Under $\QQ_n$, draw $Y_{i_1 \cdots i_p} \sim \Ber(q\, / \binom{n}{p - 1})$ independently for each $1 \leq i_1 <  \cdots < i_p \leq n$.
    \item Under $\PP_n$, first draw $\bx \in [k]^n$ uniformly at random.
        Then, draw $Y_{i_1 \cdots i_p} \sim \Ber(Q_{x_{i_1}\cdots x_{i_p}}\, / \binom{n}{p - 1})$ independently for each $1 \leq i_1 <  \cdots < i_p \leq n$.
    \end{enumerate}
    We call $\bQ$ the \emph{interaction tensor} of such a model.
\end{definition}

\begin{theorem}[Lower bound for hypergraph stochastic block model]
    \label{thm:hsbm}
    Suppose that $\bQ$ is an interaction tensor as in Definition~\ref{def:hsbm}.
    If
    \begin{equation}
        \max_{j \in \{2, \dots, k\}} |\lambda_j(\bQ[\one, \dots, \one, \cdot, \cdot])|^2 < \frac{k^{p - 1}}{p - 1} \lambda_1(\bQ[\one, \dots, \one, \cdot, \cdot]), \label{eq:hks}
    \end{equation}
    then no sequence of functions of coordinate degree $O(n\, / \log n)$ can strongly separate $\QQ_n$ from $\PP_n$ in the HSBM associated to $\bQ$.
\end{theorem}
\noindent
The threshold in \eqref{eq:hks} is again a generalized form of the KS threshold.
The hypergraph versions of these thresholds are newer and not as well-understood as the graph case in the literature, but conjectures and algorithmic results concerning such thresholds have appeared in \cite{ACKZ-2015-SpectralDetectionSparseHypergraphs,PZ-2021-CommunityDetectionSparseHSBM,CEH-2023-NonbacktrackingClusteringHypergraphs,SZ-2024-CommunityDetectionSparseHSBM}.
Evidence of statistical-to-computational gaps for certain choices of parameters has also accumulated \cite{ACKZ-2015-SpectralDetectionSparseHypergraphs,GP-2023-WeakRecoveryHSBM,GP-2024-HypergraphSBM}.
As in the graph case, Theorem~\ref{thm:hsbm} gives evidence both for the optimality of algorithms reaching this generalized KS threshold, and for the existence of statistical-to-computational gaps when inefficient algorithms are known to succeed beyond the KS threshold.

\begin{example}
    As another concrete example, consider the \emph{symmetric HSBM}, which, analogous to the symmetric SBM, has
    \[ Q_{a_1\cdots a_p} = \left\{ \begin{array}{ll} \alpha & \text{if } a_1 = \cdots = a_p, \\ \beta & \text{otherwise} \end{array}\right\}. \]
    In this case, the generalized KS condition \eqref{eq:hks} becomes:
    \[ \frac{(p - 1)(\alpha - \beta)^2}{k^{p - 1}(\alpha + (k^{p - 1} - 1)\beta)} < 1. \]
    To the best of our knowledge, even in this special case lower bounds against LDP or LCDF were not known previously for $p \geq 3$.
    As for the SBM, we give a separate and hands-on analysis of this special case in Section~\ref{sec:hsbm-warmup}.
\end{example}

Lastly, we do not give details here, but, as mentioned earlier, we discuss an example of a model with higher marginal order and therefore with a different profile of tradeoff between LCDF degree and signal-to-noise ratio, namely a model implementing random XOR-SAT formulas in a GSBM, in Section~\ref{sec:xor}.

\subsubsection{Group Problems}

In these next results, we will always consider discrete models associated to a group $G$, where we write $k = |G| > 1$ (which will be the same as the $k$ in a GSBM).
That is because our GSBMs will always be labelled by elements of $G$, and indeed our observations will also always be elements of $G$, so we will have $k = \ell = |G|$.
Both problems we study will have $p = 2$, i.e., binary interactions.

\begin{definition}[Truth-or-Haar synchronization]
    Given a group $G$ as above and a constant $\gamma > 0$, the \emph{truth-or-Haar synchronization model} associated to $G$ consists, for each $n \geq 1$, of the following pairs of probability measures over $\bY \in G^{\binom{n}{2}}$:
    \begin{enumerate}
    \item Under $\QQ_n$, draw $Y_{ij} \sim \Unif(G)$ independently for each $1 \leq i < j \leq n$.
    \item Under $\PP_n$, first draw $x_1, \dots, x_n \sim \Unif(G)$ uniformly at random.
        Then, draw $Y_{ij}$ independently for each $1 \leq i < j \leq n$ by setting it to $g_ig_j^{-1}$ with probability $\frac{\gamma}{\sqrt{n}}$ and setting it to a draw from $\Unif(G)$ with probability $1 - \frac{\gamma}{\sqrt{n}}$.
    \end{enumerate}
\end{definition}

\begin{theorem}[Lower bound for synchronization]
    \label{thm:sync}
    For any finite group $G$, if $\gamma < 1$, then no sequence of functions of coordinate degree $O(n\, / \log n)$ can strongly separate $\QQ_n$ from $\PP_n$ in the truth-or-Haar group synchronization model associated to $G$.
\end{theorem}
\noindent
This supports a prediction of \cite{Singer-2011-AngularSynchronization} and, together with the results of \cite[Section 6]{PWBM-2016-PCASpikedMatrixSynchronization}, gives evidence for a statistical-to-computational gap in these models once $|G| \geq 11$, in which case there is an inefficient algorithm that succeeds at testing for some $\gamma < 1$.

The following model is a curious variant of the above which, while its algebraic structure is quite different, at least for one class of $G$ shares precisely the same analysis as the synchronization problem.

\begin{definition}[Truth-or-Haar noisy sumset]
    Given a group $G$ as above and a constant $\gamma > 0$, the \emph{truth-or-Haar noisy sumset model} associated to $G$ is identical to the synchronization model, but with $g_ig_j^{-1}$ replaced by $g_ig_j$ (in the definition of $\PP_n$).
\end{definition}

\begin{theorem}[Lower bound for noisy sumset]
    \label{thm:sumset}
    For any finite \underline{abelian} group $G$, if $\gamma < 1$, then no sequence of functions of coordinate degree $O(n\, / \log n)$ can strongly separate $\QQ_n$ from $\PP_n$ in the truth-or-Haar noisy sumset model associated to $G$.
\end{theorem}
\noindent
The requirement that $G$ be abelian is due to the requirement of our methods that a GSBM be weakly symmetric, which otherwise would not be satisfied.
It is an interesting question whether the same result will necessarily hold for arbitrary $G$; though it seems that it should by analogy with the synchronization problem, it also seems that the analysis would require different and more problem-specific calculations.

Lastly, in Section~\ref{sec:kbk} we will discuss how our technical results can be used to sharpen the recent results of \cite{KBK-2024-LowDegreeGaussianSynchronization} on a related synchronization problem but with a quite different noise model of adding Gaussian noise to a certain matrix association to the synchronization task.
We can improve the allowable degree in this result from the original $D(n) \sim n^{1/3}$ to $D(n) \sim n$.

\subsubsection{Channel Calculus}
\label{sec:channel-calc}
Finally, in the spirit of analogous results derived in \cite{Kunisky-2024-LCDA1}, we give two general principles describing how our lower bounds against LCDF transform under certain deformations of a GSBM.

\begin{definition}[Channel resampling]
    \label{def:resampling}
    Consider a GSBM as in Definition~\ref{def:gsbm}, associated to a family of measures $(\mu_{\ba})_{\ba \in [k]^p}$ whose average is $\mu_{\avg}$.
    The \emph{$\eta$-resampling} of this GSBM is that associated instead to the family $((1 - \eta) \mu_{\ba} + \eta \mu_{\avg})_{\ba \in [k]^p}$.
    In words, $\QQ$ of the $\eta$-resampled GSBM remains the same, while in order to sample from $\PP$ under the $\eta$-resampled GSBM, we first sample from the original GSBM, and then for each observation independently with probability $\eta$ independently replace it with a draw from $\mu_{\avg}$ (which is the same as the distribution that observation would have under $\QQ$).
\end{definition}

\begin{theorem}[Resampling and characteristic tensors]
    If $\bT^{(j)}$ is a given marginalization of the characteristic tensor of an GSBM, then the same marginalization of the characteristic tensor of the $\eta$-resampling of that GSBM is $(1 - \eta)^2\,\bT^{(j)}$.
\end{theorem}

In the same vein, and directly akin to one of the results in \cite{Kunisky-2024-LCDA1} for continuous signals, we may also understand the effect of the following operation on a channel.
\begin{definition}[Channel censorship]
    \label{def:censorship}
    Consider a GSBM as in Definition~\ref{def:gsbm}, associated to a family of measures $(\mu_{\ba})_{\ba \in [k]^p}$, whose average is $\mu_{\avg}$, over a measurable space $\Omega$.
    Let ``\,$\bullet$'' denote a new symbol that does not belong to $\Omega$.
    The \emph{$\eta$-censorship} of this GSBM is that associated instead to the family $((1 - \eta) \mu_{\ba} + \eta \delta_{\bullet})_{\ba \in [k]^p}$, defined on the measurable space $\Omega \sqcup \{\bullet\}$, where $\delta_{\bullet}$ denotes the Dirac mass on this new element.
    In words, to sample from either $\QQ$ or $\PP$ under the $\eta$-censored GSBM, we first sample from the corresponding distribution in the original GSBM, and then for each observation independently with probability $\eta$ replace it with $\bullet$.
\end{definition}
\noindent
Remarkably, we find that censorship and resampling have almost the same effect on the characteristic tensor, and thus the aforementioned effective SNR, of a GSBM.
As we would expect, resampling at a given rate reduces the SNR more than censorship does, but given any rate of resampling, there is a (greater) rate of censorship that produces a model that, from the point of view of our analysis, is equivalent to the resampled one.

\begin{theorem}[Censorship and characteristic tensors]
    If $\bT^{(j)}$ is the characteristic tensor of a given marginalization of a GSBM, then the same marginalization of the characteristic tensor of the $\eta$-censorship of that GSBM is $(1 - \eta)\,\bT^{(j)}$.
\end{theorem}

Both results follow from trivially simple calculations with characteristic tensors, and we will omit the very short proofs.
Even so, in light of our general results above, these both can be useful for working with specific models.
For instance, the group synchronization and sumset models are just resamplings of a particular model where group sums or differences are observed directly without noise.
And, these results together with our analysis of SBMs allow us to immediately obtain various lower bounds for censored SBMs, as studied in \cite{ABBS-2014-ExactRecoveryCensoredSBM,SLKZ-2015-SpectralDetectionCensoredSBM}.

\section{Preliminaries}

\subsection{Notation}

The asymptotic notations $o(\cdot), O(\cdot), \omega(\cdot), \Omega(\cdot), \Theta(\cdot)$ have their usual meanings and always refer to the limit $n \to \infty$.
When these symbols have tildes on top (like $\widetilde{\Theta}$), then polylogarithmic factors in $n$ are suppressed.

For a symmetric $k \times k$ matrix $\bM$, $\lambda_1(\bM) \geq \cdots \geq \lambda_k(\bM)$ are the ordered eigenvalues of $\bM$.

$\Ber(p)$ denotes the Bernoulli probability measure of a random variable equal to 1 with probabilit $p$ and 0 with probability $1 - p$.
$\Unif(S)$ denotes the uniform probability measure over a finite set $S$.
For two probability measures $\mu, \rho$ and $\eta \in [0, 1]$, $\eta \mu + (1 - \eta) \rho$ denotes the corresponding mixture probability measure.

\subsection{Sharp Vector Bernstein Inequality}

We start with a seemingly unrelated topic, a specific version of a Bernstein-type concentration inequality over vectors.
The point of this is that we will need to understand the concentration and moments of quantities like $\|\bz - \EE \bz\|^2$, where $\bz \in \NN^k$ is the vector of numbers of elements of $[n]$ assigned each label from $[k]$ in a GSBM.
The prior work \cite{KBK-2024-LowDegreeGaussianSynchronization} which arrives at a similar issue handled this by direct combinatorial means.
We will be able to improve and simplify that analysis by instead understanding this via such a Bernstein inequality.

The kind of inequality that we need is known, if somewhat implicitly, as we will discuss below.
However, we give a self-contained proof because a rather fine detail of this particular result will be important for us.
In general, like the scalar Bernstein inequality, vector Bernstein inequalities state that $\|\sum_{i = 1}^n \bv_i\|$ for i.i.d.\ centered random vectors $\bv_i$ has subgaussian tails up to a certain size of fluctuation and exponential tails beyond that size.
It will be crucial in several of our applications (both sharpening the results of \cite{KBK-2024-LowDegreeGaussianSynchronization} and in our analysis of models with marginal order 2 and thus having sharp computational thresholds) that we have sharp control over the variance proxy in this subgaussian bound on ``somewhat large'' deviations.
As we discuss after the proof, this result achieves that, while similar results obtained by other means do not.

\begin{lemma}
    \label{lem:vec-bern}
    Let $\bv_1, \dots, \bv_n \in \RR^d$ be independent random vectors such that:
    \begin{enumerate}
    \item $\EE \bv_i = 0$ for all $i$,
    \item $\|\Cov(\bv_i)\| \leq \sigma^2$ for all $i$, and
    \item $\|\bv_i\| \leq M$ for all $i$, almost surely.
    \end{enumerate}
    Then, for any $\epsilon \in (0, 1)$, we have the tail bound
    \[ \PP\left[\left\|\sum_{i = 1}^n \bv_i \right\| \geq t\right] \leq \left(1 + \frac{2}{\epsilon}\right)^d \exp\left(-\frac{t^2}{2\frac{\sigma^2}{(1 - \epsilon)^2}n + \frac{2}{3}\frac{M}{1 - \epsilon}t}\right). \]
\end{lemma}
\begin{proof}
    Let $\sX$ be an $\epsilon$-net of the unit sphere in $\RR^d$.
    As is well-known, one may take $|\sX| \leq (1 + \frac{2}{\epsilon})^d$ (see, e.g., \cite[Corollary 4.2.13]{Vershynin-2018-HDP}).
    For any $\bv \in \RR^d$, we then have $\max_{\bx \in \sX}\langle \bv, \bx \rangle \leq \|\bv\|$, as well as, for the $\by \in \sX$ for which $\|\by - \frac{\bv}{\|\bv\|\|}\| \leq \epsilon$, that
    \begin{align*}
      \max_{\bx \in \sX}\langle \bv, \bx \rangle
      &\geq \langle \by, \bv \rangle \\
      &= \|\bv\| + \left\langle \by - \frac{\bv}{\|\bv\|}, \bv \right\rangle \\
      &\geq \|\bv\| - \epsilon \|\bv\|,
    \end{align*}
    and thus in summary we have
    \[ \max_{\bx \in \sX}\langle \bv, \bx \rangle \leq \|\bv\| \leq \frac{1}{1 - \epsilon}\max_{\bx \in \sX}\langle \bv, \bx \rangle. \]

    Taking a union bound, we then have
    \begin{align*}
      \PP\left[\left\|\sum_{i = 1}^n \bv_i \right\| \geq t\right]
      &\leq \PP\left[\max_{\bx \in \sX} \left\langle \sum_{i = 1}^n \bv_i, \bx \right\rangle \geq (1 - \epsilon) t\right] \\
      &\leq \sum_{\bx \in \sX} \PP\left[\sum_{i = 1}^n \langle \bv_i, \bx \rangle \geq (1 - \epsilon) t\right]
        \intertext{where the $\langle \bv_i, \bx \rangle$ are independent centered scalar random variables, which are bounded by $|\langle \bv_i, \bx \rangle| \leq \|\bv_i\| \leq M$ almost surely, and which have variance $\Var[\langle \bv_i, \bx \rangle] = \bx^{\top}\Cov(\bv_i)\bx \leq \|\Cov(\bv_i)\| \leq \sigma^2$. Thus, the scalar Bernstein inequality applies, giving}
      &\leq |\sX| \exp\left(-\frac{t^2}{2\frac{\sigma^2}{(1 - \epsilon)^2}n + \frac{2}{3}\frac{M}{1 - \epsilon}t}\right),
    \end{align*}
    and plugging in the bound for $|\sX|$ completes the proof.
\end{proof}

One simple other way to obtain such a result is to use symmetrization together with the Khintchine-Kahane inequality to show that $\|\sum_{i = 1}^n \bv_i\|$ has suitably decaying moments to show a similar inequality.
Another approach is to construct and analyze the Doob martingale associated to this function of the independent random variables $\bv_1, \dots, \bv_n$, using the Azuma-Hoeffding inequality or similar tools.
Both of these approaches, though, end up with a bound where $\sigma^2 = \|\Cov(\bv_i)\|$ is replaced by $\Tr(\Cov(\bv_i))$, which can be larger by up to a factor of $d$.
This would be unacceptable for our purposes, so we need specifically the above version.
Such results have appeared in the literature before; e.g., they are mentioned in \cite{WWR-2023-SelfNormalizedConcentrationVector,MTR-2024-EmpiricalBernsteinBanachSpace}, but we hope to draw attention to this argument in the above simple and straightforward setting.

\subsection{Tails and Moments of Pearson's $\chi^2$ Statistic}
\label{sec:pearson}

We now give the application of the above inequality to certain random variables constructed from multinomial vectors.
It turns out, though this connection will not play an important role for us, that these equivalently have the distribution of \emph{Pearson's $\chi^2$ statistics} from classical statistics theory.
\begin{definition}
    We denote by $\Mult(n, d)$ the multinomial distribution, the law of the random vector $(z_1, \dots, z_d) \in \ZZ_{\geq 0}^d$ giving the number of balls in each of $d$ bins after throwing $n$ balls into a bin independently and uniformly at random.
\end{definition}

\begin{definition}
    For $\bz \sim \Mult(n, d)$, we denote by $\chi^2_{\Pear}(n, d)$ the \emph{Pearson's $\chi^2$ distribution}, the law of the random variable $\frac{d}{n}\sum_{i = 1}^d (z_i - \frac{n}{d})^2 = \frac{d}{n}\|\bz - \EE \bz\|^2$.
\end{definition}

It is a classical fact following from basic vector-valued central limit theorems that, for fixed $d$ and $n \to \infty$, we have the convergence in distribution
\[ \sqrt{\frac{d}{n}} \left(\bm z - \frac{n}{d}\one\right) \xto{\text{(law)}} \sN\left(\bm 0, \bm I_d - \frac{1}{d}\one_d\one_d^{\top}\right), \]
where the covariance is the orthogonal projection to the orthogonal complement of the $\one_d$ direction.
In particular, $X \sim \chi^2_{\Pear}$ then converges in distribution to $\chi^2(d - 1)$, the $\chi^2$ law with $d - 1$ degrees of freedom.

Our goal here will be to show that, even non-asymptotically, the tails and moments of $\chi^2_{\Pear}$ resemble those of $\chi^2(d - 1)$.
While the arguments are simple once we are aware of the sharp vector Bernstein inequality of Lemma~\ref{lem:vec-bern}, to the best of our knowledge these kinds of results have not appeared in the statistics literature before, even though the accuracy of the approximation $\chi^2_{\Pear} \approx \chi^2(d - 1)$ has been discussed at length \cite{GGM-1970-NumericalDistributionPearson,LSW-1984-MomentsPearsonDistribution,HG-1986-PoissonChiSquaredPearson}.

\begin{lemma}[Right tail of $\chi^2_{\Pear}$]
    \label{lem:pear-right-tail}
    For any $n, d \geq 1$, $t \geq 0$, and $\epsilon \in (0, 1)$,
    \[ \Px_{X \sim \chi^2_{\Pear}(n, d)}[X \geq t] \leq \left(1 + \frac{2}{\epsilon}\right)^d \exp\left(-\frac{\frac{1}{2}(1 - \epsilon)^2t}{1 + \frac{1 - \epsilon}{3}\sqrt{\frac{d - 1}{n}}\sqrt{t}}\right). \]
\end{lemma}
\noindent
We note that when $d$ is fixed, $\epsilon$ is small, and $t = o(n)$, then we approximately recover the tail behavior of a $\chi^2$ distribution, whose density is proportional to $\exp(-x / 2)$ up to polynomial factors in $x$.
\begin{proof}
    Write $\be_1, \dots, \be_d \in \RR^d$ for the standard basis of $\RR^d$, and let $i_1, \dots, i_n \sim \Unif([d])$.
    Letting $\bm z \colonequals \sum_{a = 1}^n \be_{i_a}$, we have that the law of $\bm z$ is $\Mult(n, d)$.
    This realizes the multinomial distribution as a sum of i.i.d.\ random vectors, to which we will apply our vector Bernstein inequality.

    In particular, writing $\bv_a \colonequals \be_{i_a} - \frac{1}{d}\one_d = \be_{i_a} - \EE[\be_{i_a}]$, we have that $\frac{d}{n}\|\sum_{a = 1}^n \bv_a\|^2$ has the law $\chi^2_{\Pear}(n, d)$.
    We also have $\EE \bv_a = \bm 0$, $\Cov(\bv_a) = \frac{1}{d} \bm I_d - \frac{1}{d^2}\one_d\one_d^{\top}$ whereby $\|\Cov(\bv_a)\| = \frac{1}{d}$, and $\|\bv_a\|^2 = 1 - \frac{1}{d} = \frac{d - 1}{d}$ almost surely.

    Plugging this information into Lemma~\ref{lem:vec-bern}, we find that, for any $\epsilon \in (0, 1)$,
    \begin{align*}
      \Px_{X \sim \chi^2_{\Pear}(n, d)}[\sqrt{X} \geq t]
      &= \PP\left[\left\|\sum_{a = 1}^n \bv_a\right\| \geq \sqrt{\frac{n}{d}} t \right] \\
      &\leq \left(1 + \frac{2}{\epsilon}\right)^d \exp\left(-\frac{\frac{n}{d}t^2}{2\frac{1}{(1 - \epsilon)^2}\frac{n}{d} + \frac{2}{3}\frac{1}{1 - \epsilon}\sqrt{\frac{n}{d}}\sqrt{\frac{d - 1}{d}}t}\right),
    \end{align*}
    and rearranging gives the result.
\end{proof}

\begin{corollary}[Moments of $\chi^2_{\Pear}$]
    For any $n, d \geq 1$ integers, $r \geq 1$, and $\delta, \epsilon \in (0, 1)$,
    \[ \Ex_{X \sim \chi^2_{\Pear}(n, d)}[X^r] \leq 2r\left(1 + \frac{2}{\epsilon}\right)^d\left[ \left(\frac{1 + \sqrt{\delta d}}{(1 - \epsilon)^2}\right)^r2^r \Gamma(r) + \left(\frac{\frac{4}{\delta} + 4d}{(1 - \epsilon)^2}\right)^r \frac{\Gamma(2r)}{n^r}\right]. \]
\end{corollary}
\begin{proof}
    We start by integrating the tail bound:
    \begin{align*}
      \Ex_{X \sim \chi^2_{\Pear}(n, d)}[X^r]
      &= \int_0^{\infty} \PP[X^r > t]\,dt \\
      &= \int_0^{\infty} \PP[\sqrt{X} > t^{\frac{1}{2r}}]\,dt \\
      &\leq \left(1 + \frac{2}{\epsilon}\right)^d \int_0^{\infty} \exp\left(-\frac{\frac{1}{2}(1 - \epsilon)^2t^{\frac{1}{r}}}{1 + \frac{1 - \epsilon}{\sqrt{2}}\sqrt{\frac{d}{n}}t^{\frac{1}{2r}}}\right)\,dt \\
      &= \left(1 + \frac{2}{\epsilon}\right)^d \cdot r\left(\frac{2}{(1 - \epsilon)^2}\right)^r \int_0^{\infty} \exp\left(-\frac{s}{1 + \sqrt{\frac{d}{n}s}}\right)s^{r - 1}\,ds.
    \end{align*}
    For the remaining integral, let us fix $\delta > 0$ as in the statement, and break the integral up into two parts at $\delta n$:
    \begin{align*}
      &\int_0^{\infty} \exp\left(-\frac{s}{1 + \sqrt{\frac{d}{n}s}}\right)s^{r - 1}\,ds \\
      &= \int_0^{\delta n} \exp\left(-\frac{s}{1 + \sqrt{\frac{d}{n}s}}\right)s^{r - 1}\,ds + \int_{\delta n}^{\infty} \exp\left(-\frac{s}{1 + \sqrt{\frac{d}{n}s}}\right)s^{r - 1}\,ds \\
      &\leq \int_0^{\infty} \exp\left(-\frac{s}{1 + \sqrt{\delta d}}\right)s^{r - 1}\,ds + \int_{0}^{\infty} \exp\left(-\frac{\sqrt{s}}{\sqrt{\frac{1}{\delta n}} + \sqrt{\frac{d}{n}}}\right)s^{r - 1}\,ds
        \intertext{and both remaining integrals may be viewed as evaluations of the $\Gamma$ function, the second one after a further substitution $s = r^2$, giving}
      &= (1 + \sqrt{\delta d})^r\,\Gamma(r) + 2\left(\frac{1}{\sqrt{\delta}} + \sqrt{d}\right)^{2r} \frac{\Gamma(2r)}{n^r}.
    \end{align*}
    The result as stated is obtained by a few elementary bounds to simplify this expression.
\end{proof}

The formula above is complicated but the point is simple: when $\delta, \epsilon$ are both small and $r = o(n)$, then the prefactors and the second term are negligible, and the expression becomes $O_d(r2^r\Gamma(r))$, which is same scale as the $r$th moment of $\chi^2(d - 1)$.
We express this in the following corollary before giving the proof.
\begin{corollary}[Simplified moment bound]
    \label{cor:simple-moment}
    For all $\epsilon > 0$, there exist constants $C, \gamma > 0$ such that, for all $n \geq 1$ and $1 \leq r \leq \gamma n$, we have
    \[ \Ex_{X \sim \chi^2_{\Pear}(n, d)}[X^r] \leq r^{3/2} C^d \left(\frac{(2 + \epsilon)r}{e}\right)^r. \]
\end{corollary}
\begin{proof}
    Choosing $\delta, \epsilon$ in the previous Corollary appropriately and using Stirling's approximation $\Gamma(r) \lesssim \sqrt{r} (r / e)^r$, we find that there exists $C$ such that, for all $n \geq 1$ an integer and $r \geq 1$ real, we have
    \[ \Ex_{X \sim \chi^2_{\Pear}(n, d)}[X^r] \leq r^{3/2}C^d\left[ \left(\frac{(2 + \epsilon)r}{e}\right) + \left(\frac{Cr^2}{n}\right)^r \right]. \]
    But now, choosing $\gamma \leq \frac{2}{Ce}$, if $r \leq \gamma n$, then $r / n \leq \gamma$, so the second term above is smaller than the first, and the result follows (taking $C$ slightly larger to absorb a factor of 2).
\end{proof}

\subsection{Tools for Overlap Form of Low Degree Advantages}

Next, we provide some general tools for working with expressions that often appear in the style of analysis of LCDF and LDP algorithms that seeks to reduce the task to questions about a single ``overlap'' random variable involving the similarity of two draws of a hidden signal.
This theme has been explored in previous work including \cite{BKW-2019-ConstrainedPCA,KWB-2022-LowDegreeNotes, BBKMW-2020-SpectralPlantingColoring, Kunisky-2020-LowDegreeMorris, Kunisky-2024-LCDA1}, following similar observations for computations of $\chi^2$ divergences in \cite{MRZ-2015-LimitationsSpectral, PWBM-2016-PCASpikedMatrixSynchronization,BMVVX-2018-InfoTheoretic}.
We aim here to distill an elementary but important manipulation that often appears in these proofs, to let it be used more flexibly in our context and related ones in the future.

\begin{definition}[Truncated exponential]
    For $D \geq 0$ an integer and $t \in \RR$, we define
    \[ \exp^{\leq D}(t) \colonequals \sum_{d = 0}^D \frac{t^d}{d!}. \]
\end{definition}

\begin{proposition}[Basic properties]
    \label{prop:exp-trunc-bound}
    For any $t \in \RR$ and $D \geq 0$ an integer,
    \begin{align*}
      \exp^{\leq D}(t)
      &\leq \exp^{\leq D}(|t|) \\
      &\leq 2 \cdot \frac{(2D \vee |t|)^D}{D!}. \numberthis \label{eq:exp-trunc-bound}
    \end{align*}
    Further, $\exp^{\leq D}(t)$ is an increasing function on $t \in \RR_{\geq 0}$.
\end{proposition}
\noindent
All of these observations but the bound \eqref{eq:exp-trunc-bound} are immediate, and that bound is proved in Corollary~5.2.3 of \cite{Kunisky-2021-SpectralBarriersCertification}.

The techniques of the works cited above lead to expressions of the form $\EE \exp^{\leq D(n)}(R_n)$ for a sequence of scalar random variables $R_n \geq 0$ and a sequence of growing degrees $D(n)$.
Indeed, often $R_n$ even converge in distribution to a limiting random variable, but the challenge is to understand the ``race'' between this convergence and the convergence of the function $\exp^{\leq D(n)}(t)$ to $\exp(t)$.
The following gives a general treatment of precisely this situation, with no reference to low degree analysis.
We will use the following non-asymptotic form of Stirling's approximation.

\begin{proposition}
    \label{prop:factorial}
    For any $d \geq 1$, $d! \geq (d / e)^d$.
\end{proposition}

\begin{lemma}
    \label{lem:overlap}
    Let $R_n \geq 0$ be a sequence of real bounded random variables and $D(n) \in \NN$.
    Suppose there exists a sequence $A(n) \in \RR_{\geq 0}$ such that the following conditions hold:
    \begin{enumerate}
    \item For all sufficiently large $n$,
        \[ A(n) \geq D(n) \left(2 \vee \log\left(\frac{\|R_n\|_{\infty}}{D(n)}\right)\right). \]
        (Recall that $\|R_n\|_{\infty}$ is the smallest $C > 0$ such that $R_n \leq C$ almost surely.)
    \item For bounded $f: \RR_{\geq 0} \to \RR_{\geq 0}$ such that $\int_0^{\infty}f(t)\,dt < \infty$, for sufficiently large $n$, for all $t \in [0, A(n)]$,
        \[ \PP[R_n \geq t] \leq f(t) \exp(-t). \]
    \end{enumerate}
    Then, we have
    \[ \limsup_{n \to \infty} \EE \exp^{\leq D(n)}(R_n) < \infty. \]
\end{lemma}
\begin{proof}
    We decompose
    \[ \EE \exp^{\leq D(n)}(R_n) = \underbrace{\EE \One\{R_n > A(n)\} \exp^{\leq D(n)}(R_n)}_{\equalscolon E_{\tlarge}} + \underbrace{\EE \One\{R_n \leq A(n)\}\exp^{\leq D(n)}(R_n)}_{\equalscolon E_{\tsmall}}. \]

    For the ``large deviations'' term $E_{\tlarge}$, we use that $\exp^{\leq D(n)}(r)$ is monotone in $r$ and bound using our assumption on $\|R_n\|$ and Proposition~\ref{prop:exp-trunc-bound}
    \begin{align*}
      E_{\tlarge}
      &\leq \PP[R_n > A(n)] \exp^{\leq D(n)}(\|R_n\|_{\infty}) \\
      &\leq 2f(A(n))\exp(-A(n)) \frac{(2D(n) \vee \|R_n\|_{\infty})^{D(n)}}{D(n)!}
        \intertext{and using Proposition~\ref{prop:factorial} we have}
      &\leq 2f(A(n))\exp(-A(n)) \left(2e \vee \frac{\|R_n\|_{\infty}}{D(n)}\right)^{D(n)} \\
      &\leq 2f(A(n))\exp\left(-A(n) + D(n) \log\left(2e \vee \frac{\|R_n\|_{\infty}}{D(n)}\right)\right),
    \end{align*}
    and thus $E_{\tlarge}$ is bounded since by our assumptions $f$ is bounded and the expression in the exponential is negative.

    For the ``small deviations'' term $E_{\tsmall}$, we use that $\exp^{\leq D}(r) \leq \exp(r)$, so
    \begin{align*}
      E_{\tsmall}
      &\leq \EE \One\{R_n \leq A(n)\} \exp(R_n) \\
      &= \int_0^{\infty} \PP[\One\{R_n \leq A(n)\} \exp(R_n) \geq u]\, du \\
      &= \int_0^{\infty} \PP[\One\{R_n \leq A(n)\} \exp(R_n) \geq \exp(t)] \cdot \exp(t)\,dt \\
      &= \int_0^{A(n)} \PP[R_n \geq t] \cdot \exp(t)\,dt \\
      &\leq \int_0^{A(n)} f(t)\, dt \\
      &\leq \int_0^{\infty} f(t)\, dt,
    \end{align*}
    which is again a finite constant, so $E_{\tsmall}$ is bounded.
    Combining the two gives the result.
\end{proof}

The prototypical applications in, e.g., \cite{PWBM-2016-PCASpikedMatrixSynchronization, KWB-2022-LowDegreeNotes} to spiked matrix models, have used this as follows: for a constant $\lambda > 0$, $R_n$ is the law of $\frac{\lambda^2}{n} \langle \bx^{(1)}, \bx^{(2)} \rangle^2$ for, say, $\bx^{(i)} \sim \Unif(\{\pm 1\}^n)$ independently (or, more generally, vectors with i.i.d.\ bounded entries having mean zero and unit variance) and $0 < \lambda < 1$.
Then, $\|R_n\|_{\infty} = O(n)$ since $R_n \leq \frac{\lambda^2}{n} \|\bx^{(1)}\|^2 \|\bx^{(2)}\|^2$, and the conditions of the Lemma hold with $f(t) = C\exp(-\delta t)$ for large $C > 0$ and small $\delta > 0$ (depending on $\lambda$, which is where it is important that $\lambda < 1$) and $A(n) = \epsilon n$ for small $\epsilon > 0$.
The tail bound on $R_n$ in Condition 2 of the Lemma is what those works call a ``local Chernoff bound'' for the inner product $\langle\bx^{(1)}, \bx^{(2)} \rangle$.
This result then gives low-degree lower bounds for, say, any $D(n) = o(n / \log n)$.

We are just observing here that all of these details of their analysis are incidental, and the above is the distilled analytic content of the argument.
This is a simple observation, but, given the general bounds developed recently in \cite{Kunisky-2024-LCDA1} leading to more unusual forms of the random variables $R_n$, it seems that this will be a useful tool, and it will be helpful in our setting already.
For instance, we will encounter a situation where $R_n$ decomposes as a sum of (dependent) random variables, but we may establish the conditions of Lemma~\ref{lem:overlap} simply by a carefully tuned union bound.

\section{Main Lower Bounds}

\subsection{Coordinate Advantage Bound}

We will use the following object to show that strong separation is impossible, following a style of argument that has appeared recently in \cite{COGHWZ-2022-PhaseTransitionsGroupTesting,BAHSWZ-2022-FranzParisiLowDegree} using such quantities to reason about separation criteria.
\begin{definition}[Coordinate advantage]
    \label{def:cadv}
    For $\PP, \QQ$ as in Definition~\ref{def:gsbm} and $D \geq 1$, we define
    \begin{equation}
    \CAdv_{\leq D}(\PP, \QQ) \colonequals \left\{\begin{array}{ll} \text{maximize} & \Ex_{\by \sim \PP} f(\by) \\ \text{subject to} & \Ex_{\by \sim \QQ} f(\by)^2 \leq 1, \\ & \cdeg(f) \leq D \end{array}\right\}.
\end{equation}
\end{definition}
\noindent
The following is then immediate from the definition of strong separation.
\begin{proposition}
    \label{prop:cadv-sep}
    For a sequence of pairs of probability measures $\QQ_n$, $\PP_n$ and $D(n) \in \NN$, if $\CAdv_{\leq D(n)}(\PP_n, \QQ_n) = O(1)$ as $n \to \infty$, then there is no sequence of functions of coordinate degree at most $D(n)$ that strongly separate $\QQ_n$ from $\PP_n$.
\end{proposition}

We follow the approach of \cite{Kunisky-2024-LCDA1} in deriving tractable bounds on the coordinate advantage for GSBMs.
Consider first a single GSBM rather than a sequence, with parameters $p, k, n$ as in Definition~\ref{def:gsbm} and with a resulting pair of probability measures $\QQ$ and $\PP$.
We begin by deriving the following remarkably simple bound on the coordinate advantage involving the characterstic tensor and a multinomial random vector.
\begin{lemma}
    \label{lem:cadv}
    In the above setting, let $\bT$ be the characteristic tensor of the GSBM.
    Then,
    \[ \CAdv_{\leq D}(\QQ, \PP)^2 \leq \Ex_{\bm z \sim \Mult(n, k^2)} \exp^{\leq D}\left(\langle \bT, \bm z^{\otimes p} \rangle\right). \]
\end{lemma}
\noindent
Note that this Lemma achieves a dramatic dimensionality reduction, similar to the ``overlap formulas'' discussed earlier: while our original problem had dimension growing with $n$, $\bT$ is a tensor of size not depending on $n$, so we have isolated the role of $n$ to its participation in the multinomial distribution of the vector $\bm z$ (which, like $\bT$, has fixed dimension).

\begin{proof}
    Following \cite{Kunisky-2024-LCDA1}, we define the \emph{channel overlap}, for $\ba = (a_1, \dots, a_p), \bb = (b_1, \dots, b_p) \in [k]^p$, to be the quantity:
    \[ R(\ba, \bb) \colonequals \Ex_{y \sim \mu_{\avg}}\left[\left(\frac{d\mu_{\ba}}{d\mu_{\avg}}(y) - 1\right)\left(\frac{d\mu_{\bb}}{d\mu_{\avg}}(y) - 1\right)\right]. \]
    Recall that these are also the entries of $\bT$.
    By our assumption of weak symmetry, $R$ is unchanged by a simultaneous permutation of both of its inputs: for any $\sigma \in \Sym([p])$,
    \[ R((a_{\sigma(1)}, \dots, a_{\sigma(p)}), (b_{\sigma(1)}, \dots, b_{\sigma(p)})) = R((a_1, \dots, a_p), (b_1, \dots, b_p)). \]

    From this we build a ``total overlap,'' which is just the sum of the channel overlap over all observations.
    That is, for $\bx^{(1)}, \bx^{(2)} \in [k]^n$, we set
    \begin{align*}
      R(\bx^{(1)}, \bx^{(2)})
      &\colonequals \sum_{1 < i_1 < \cdots < i_p \leq n} R((x_{i_1}^{(1)}, \dots, x_{i_p}^{(1)}), (x_{i_1}^{(2)}, \dots, x_{i_p}^{(2)})) \\
      &= \sum_{1 < i_1 < \cdots < i_p \leq n} \Ex_{y \sim \mu_{\avg}}\left[\left(\frac{d\mu_{(x_{i_1}^{(1)}, \dots, x_{i_p}^{(1)})}}{d\mu_{\avg}}(y) - 1\right)\left(\frac{d\mu_{(x_{i_1}^{(2)}, \dots, x_{i_p}^{(2)})}}{d\mu_{\avg}}(y) - 1\right)\right] \\
      &= \frac{1}{p!} \sum_{\substack{\bm i \in [n]^p \\ \text{all entries distinct}}} \Ex_{y \sim \mu_{\avg}}\left[\left(\frac{d\mu_{(x_{i_1}^{(1)}, \dots, x_{i_p}^{(1)})}}{d\mu_{\avg}}(y) - 1\right)\left(\frac{d\mu_{(x_{i_1}^{(2)}, \dots, x_{i_p}^{(2)})}}{d\mu_{\avg}}(y) - 1\right)\right],
    \end{align*}
    where the last step follows by our observation from weak symmetry above.

    By \cite[Theorem 3.5]{Kunisky-2024-LCDA1}, the low coordinate degree advantage is bounded in terms of this total overlap as:
    \[ \CAdv_{\leq D}(\QQ, \PP)^2 \leq \Ex_{\bx^{(1)}, \bx^{(2)} \sim \Unif([k]^p)} \exp^{\leq D}(R(\bx^{(1)}, \bx^{(2)})). \]
    Consider next the modified overlap allowing for all tuples of indices:
    \[ R^{\prime}(\bx^{(1)}, \bx^{(2)}) \colonequals \frac{1}{p!} \sum_{\bm i \in [n]^p} \Ex_{y \sim \mu_{\avg}}\left[\left(\frac{d\mu_{(x_{i_1}^{(1)}, \dots, x_{i_p}^{(1)})}}{d\mu_{\avg}}(y) - 1\right)\left(\frac{d\mu_{(x_{i_1}^{(2)}, \dots, x_{i_p}^{(2)})}}{d\mu_{\avg}}(y) - 1\right)\right]. \]
    We claim that our bound only increases upon replacing $R$ with $R^{\prime}$:
    \begin{equation}
        \Ex_{\bx^{(1)}, \bx^{(2)} \sim \Unif([k]^p)} \exp^{\leq D}(R(\bx^{(1)}, \bx^{(2)})) \leq \Ex_{\bx^{(1)}, \bx^{(2)} \sim \Unif([k]^p)} \exp^{\leq D}(R^{\prime}(\bx^{(1)}, \bx^{(2)})). \label{eq:overlap-comparison}
    \end{equation}
    Indeed, we may write the two overlaps as
    \begin{align*}
      R(\bx^{(1)}, \bx^{(2)}) &= \frac{1}{p!} \sum_{\substack{\bm i \in [n]^p \\ \text{all entries distinct}}} \left\langle \bar{L}_{(x_{i_1}^{(1)}, \dots, x_{i_p}^{(1)})}, \bar{L}_{(x_{i_1}^{(2)}, \dots, x_{i_p}^{(2)})} \right\rangle, \\
      R^{\prime}(\bx^{(1)}, \bx^{(2)}) &= \frac{1}{p!} \sum_{\bm i \in [n]^p} \left\langle \bar{L}_{(x_{i_1}^{(1)}, \dots, x_{i_p}^{(1)})}, \bar{L}_{(x_{i_1}^{(2)}, \dots, x_{i_p}^{(2)})} \right\rangle,
    \end{align*}
    where $\bar{L}_{\ba} = \frac{d\mu_{\ba}}{d\mu_{\avg}} - 1$ is the centered likelihood ratio and the inner products are in $L^2(\mu_{\avg})$.
    These likelihood ratios satisfy
    \[ \Ex_{\bx^{(j)} \sim \Unif([k]^n)} \bar{L}_{(x_{i_1}^{(j)}, \dots, x_{i_p}^{(j)})} = 0 \]
    for all $\bm i \in [k]^n$ and $j \in \{1, 2\}$.
    Therefore, upon fully expanding the powers of $R$ and $R^{\prime}$ appearing on either side of \eqref{eq:overlap-comparison} and taking the expectation over $\bx^{(1)}, \bx^{(2)}$ first, we will have that only the diagonal terms contribute to such an expansion, and each contribution is non-negative.
    There are more such terms in the expression with $R^{\prime}$, and so \eqref{eq:overlap-comparison} holds (see \cite[Lemma 4.5]{Kunisky-2024-LCDA1} for full details of this argument).
    Thus we also have
    \[ \CAdv_{\leq D}(\QQ, \PP)^2 \leq \Ex_{\bx^{(1)}, \bx^{(2)} \sim \Unif([k]^p)} \exp^{\leq D}(R^{\prime}(\bx^{(1)}, \bx^{(2)})). \]

    Now, define
    \[ z_{a, b} = z_{a, b}(\bx^{(1)}, \bx^{(2)}) \colonequals \#\{i \in [n]: x^{(1)}_i = a, x^{(2)}_i = b\}. \]
    First note that we may rewrite $R^{\prime}$ by grouping like terms as
    \begin{align*}
      R^{\prime}(\bx^{(1)}, \bx^{(2)})
      &= \sum_{\ba \in [k]^p} \sum_{\bm b \in [k]^p} \frac{1}{p!}\Ex_{y \sim \mu_{\avg}}\left[\left(\frac{d\mu_{\ba}}{d\mu_{\avg}}(y) - 1\right)\left(\frac{d\mu_{\bb}}{d\mu_{\avg}}(y) - 1\right)\right] z_{a_1, b_1} \cdots z_{a_p, b_p} \\
      &= \langle \bT, \bz^{\otimes p} \rangle,
    \end{align*}
    where $\bT$ is the characteristic tensor.
    Finally, the proof is complete after observing that, when $\bx^{(1)}, \bx^{(2)} \sim \Unif([k]^p)$ independently, then the law of $\bm z$ is precisely $\Mult(n, k^2)$, but where in the expression above we rearrange its entries into a $k \times k$ matrix.
\end{proof}

Now, note that $\EE \bz = \frac{n}{k^2} \one$, for $\one$ the all-ones vector (of dimension $k^2$, but implicitly indexed by $[k] \times [k]$).
We define for $\bm z$ as in the Lemma its centered version,
\[ \bbz \colonequals \bz - \frac{n}{k^2} \one, \]
whereby we have
\[ \bz = (\bz - \EE \bz) + \EE \bz = \bbz + \frac{n}{k^2} \one. \]
The following straightforward calculation describes how the marginal characteristic tensors arise from plugging the above observation into the result of the Lemma.
\begin{proposition}
    \label{prop:marginal-expansion}
    In the above setting, if the marginal order of the GSBM is at least $p_*$, then
    \[ \langle \bT, \bz^{\otimes p} \rangle = \sum_{j = p_*}^{p} \binom{p}{j} n^{p - j} \langle \bT^{(j)}, \bbz^{\otimes (j)}\rangle. \]
\end{proposition}
\begin{proof}
    Observe that $\bT$ is a symmetric tensor.
    We then expand directly, grouping like terms:
    \begin{align*}
      \langle \bT, \bz^{\otimes p} \rangle
      &= \left\langle \bT, \left(\bbz + \frac{n}{k^2} \one\right)^{\otimes p} \right\rangle
        \intertext{where since $\bT$ is symmetric, in each of the $2^p$ terms arising we may group together all of the occurrences of $\one$ and all of the occurrences of $\bbz$, giving an expansion as in the binomial theorem,}
      &= \sum_{j = 0}^p \binom{p}{j} n^{p - j} \frac{1}{k^{2(p - j)}} \langle \bT, \underbrace{\one \otimes \cdots \otimes \one}_{p - j \text{ times}} \otimes \underbrace{\bbz \otimes \cdots \otimes \bbz}_{j \text{ times}} \rangle
        \intertext{but now by definition of the marginalized tensors,}
      &= \sum_{j = 0}^p \binom{p}{j} n^{p - j} \langle \bT^{(j)}, \bbz^{\otimes (p - j)}\rangle,
    \end{align*}
    giving the result since all terms with $j < p^*$ vanish by the definition of marginal order.
\end{proof}

Finally, what we will actually use is the following bound on this overlap and therefore on the advantage.
For the sake of clarity, we give both the form involving the centered multinomial vector $\bar{\bz}$ following from the above, and the normalized version in terms of the Pearson $\chi^2$ statistic that will connect to the bounds we have developed in Section~\ref{sec:pearson}.
\begin{corollary}
    \label{cor:cadv-bound}
    In the above setting, if the marginal order of the GSBM is at least $p_*$, then
    \begin{align*}
      \CAdv_{\leq D}(\QQ, \PP)
      &\leq \Ex_{\bm z \sim \Mult(n, k^2)} \exp^{\leq D}\left(\sum_{j = p_*}^{p} \binom{p}{j} \|\bT^{(j)}\|_{\inj} \, n^{p - j} \|\bbz\|^{j}\right) \\
      &= \Ex_{X \sim \chi^2_{\Pear}(n, k^2)} \exp^{\leq D}\left(\sum_{j = p_*}^{p} \binom{p}{j} \frac{1}{k^j} \|\bT^{(j)}\|_{\inj} \, n^{p - j/2} X^{j/2}\right).
    \end{align*}
\end{corollary}
\begin{proof}
    The first part follows by combining Proposition~\ref{prop:exp-trunc-bound}, Proposition~\ref{prop:marginal-expansion}, and Lemma~\ref{lem:cadv}.
    The second part follows from observing that the law of $\frac{k^2}{n} \|\bar{\bz}\|^2$ is precisely $\chi^2_{\Pear}(n, k^2)$.
\end{proof}
\noindent
By way of intuition, in the final expression, note that we expect $X = \Theta(1)$ typically as $n \to \infty$.
Thus, by far the largest term in the truncated exponential will correspond to the smallest value of $j$, which is $j = p_*$.
So, this bound indeed expresses that the marginal order of the GSBM governs the low coordinate degree advantage.

Before proceeding, we also establish the following bound that will be useful throughout.
\begin{proposition}
    \label{prop:chi2-pear-infty}
    Let $X \sim \chi_{\Pear}^2(n, k^2)$.
    Then, $\|X\|_{\infty} \leq k^2n$.
\end{proposition}
\begin{proof}
    Recall that we have $X = \frac{k^2}{n}\sum_{i = 1}^{k^2} (z_i - \frac{n}{k^2})^2$ for $\bz \sim \Mult(n, k^2)$.
    Since $\sum_{i = 1}^{k^2} z_i = n$, we may interpret $\sum_{i = 1}^{k^2} (z_i - \frac{n}{k^2})^2$ as a variance, and thus bound
    \begin{align*}
      \sum_{i = 1}^{k^2} \left(z_i - \frac{n}{k^2}\right)^2
      &\leq \sum_{i = 1}^{k^2} z_i^2 \\
      &\leq \left(\sum_{i = 1}^{k^2} z_i\right)^2 \\
      &\leq n^2,
    \end{align*}
    and the result follows from substituting this into the expression for $X$.
\end{proof}

\subsection{Proof of Theorem~\ref{thm:lcdf-p3}: General Marginal Order}

\begin{proof}[Proof of Theorem~\ref{thm:lcdf-p3}]
We allow $C$ to be a constant changing from line to line, which always needs only to be taken sufficiently large depending on the parameters $p, k, \ell$ (which do not change with $n$).
Let $X_n \sim \chi^2_{\Pear}(n, k^2)$.
Using Proposition~\ref{prop:chi2-pear-infty}, we may bound crudely, for all $p_* \leq j \leq p$,
\begin{align*}
  R_{n, j}
  &\colonequals \binom{p}{j} \frac{1}{k^j} \|\bT^{(j)}_n\|_{\inj} \, n^{p - \frac{j}{2}} X_n^{\frac{j}{2}} \\
  &\leq C \|\bT^{(j)}_n\|_{\inj} \, n^{p - \frac{p_*}{2}} X_n^{\frac{p_*}{2}}, \\
  R_n
  &\colonequals \sum_{j = 2}^{p} R_{n, j} \\
  &\leq C \left(\max_{p_* \leq j \leq p} \|\bT^{(j)}_n\|_{\inj}\right) n^{p - \frac{p_*}{2}} X_n^{\frac{p_*}{2}}.
\end{align*}
Using Corollary~\ref{cor:cadv-bound}, we have
\begin{align*}
  \CAdv_{\leq D(n)}(\QQ_n, \PP_n)^2
  &\leq \EE \exp^{\leq D(n)}(R_n) \\
  &= \sum_{d = 0}^{D(n)} \frac{1}{d!} \EE R_n^d \\
  &\leq \sum_{d = 0}^{D(n)} \frac{1}{d!} \left(C \left(\max_{p_* \leq j \leq p} \|\bT^{(j)}_n\|_{\inj}\right) n^{p - \frac{p_*}{2}}\right)^d \EE X_n^{\frac{p_*d}{2}}
    \intertext{and now, provided that $D(n) \leq n / C$, we have by Corollary~\ref{cor:simple-moment}}
  &\leq C \sum_{d = 0}^{D(n)} \frac{1}{d!} \left(C \left(\max_{p_* \leq j \leq p} \|\bT^{(j)}_n\|_{\inj}\right) n^{p - \frac{p_*}{2}}\right)^d d^{\frac{3}{2}} (p_*d)^{\frac{p_*d}{2}} \\
  &\leq C \sum_{d = 0}^{D(n)} \frac{1}{d!} \left(C\left(\max_{p_* \leq j \leq p} \|\bT^{(j)}_n\|_{\inj}\right) n^{p - \frac{p_*}{2}}d^{\frac{p_*}{2}}\right)^d
    \intertext{and by Proposition~\ref{prop:factorial},}
  &\leq C \sum_{d = 0}^{D(n)} \left(C\left(\max_{p_* \leq j \leq p} \|\bT^{(j)}_n\|_{\inj}\right) n^{p - \frac{p_*}{2}}d^{\frac{p_*}{2} - 1}\right)^d \\
  &\leq C \sum_{d = 0}^{\infty} \left(C\left(\max_{p_* \leq j \leq p} \|\bT^{(j)}_n\|_{\inj}\right) n^{p - \frac{p_*}{2}}D(n)^{\frac{p_*}{2} - 1}\right)^d,
\end{align*}
and the result follows since our assumption implies that the inner expression is, say, at most $1/2$ for sufficiently large $n$, whereby this series converges.
\end{proof}

\subsection{Proof of Theorem~\ref{thm:lcdf-p2}: Marginal Order 2}

\begin{proof}[Proof of Theorem~\ref{thm:lcdf-p2}]
    We will at first follow the same ideas as the previous proof, but eventually will need to be much more precise.
    Let $X_n \sim \chi^2_{\Pear}(n, k^2)$, and define
    \begin{align*}
      R_{n, j} &\colonequals \binom{p}{j} \frac{1}{k^j} \|\bT^{(j)}_n\|_{\inj} \, n^{p - \frac{j}{2}} X_n^{\frac{j}{2}} \text{ for } 2 = p_* \leq j \leq p, \\
      R_n &\colonequals \sum_{j = 2}^{p} R_{n, j}.
    \end{align*}
    Then, Corollary~\ref{cor:cadv-bound} states that
    \[ \CAdv_{\leq D(n)}(\QQ_n, \PP_n) \leq \EE \exp^{\leq D(n)}(R_n). \]
    We will use Lemma~\ref{lem:overlap}, which concerns precisely such expressions, to show that this is bounded as $n \to \infty$.

    We first establish some preliminary bounds using the assumptions of the Theorem.
    Recall that we assume, for constants $C > 0$ and $\epsilon \in (0, 1)$ and sufficiently large $n$, that
    \begin{align*}
      \|\bT^{(2)}_n\| &\leq (1 - \epsilon) \frac{k^2}{p(p - 1)} \frac{1}{n^{p - 1}}, \\
      \|\bT^{(j)}_n\|_{\inj} &\leq C\frac{1}{n^{p - 1}} \text{ for all } 3 \leq j \leq p.
    \end{align*}
    Let us slightly abuse notation and let the constant $C$ from the statement of the theorem increase from line to line in the proof, but only depending on the constant parameters $p, k, \ell, \epsilon$ from the statement of the Theorem (not on the growing parameter $n$).
    Plugging this in, we then find
    \begin{align}
      R_{n, 2} &\leq (1 - \epsilon) \frac{1}{2} X_n, \label{eq:Rn2-estimate} \\
      R_{n, j} &\leq C n^{-\frac{j}{2} + 1} X_n^{\frac{j}{2}} \text{ for all } 3 \leq j \leq p. \label{eq:Rn3-estimate}
    \end{align}

    Now, we prove a tail bound for $R_n$ in order to apply Lemma~\ref{lem:overlap}.
    By the union bound and the above observations,
    \begin{align*}
      \PP[R_n \geq t]
      &\leq \PP\left[R_{n, 2} \geq \left(1 - \frac{\epsilon}{2}\right)t\right] + \sum_{j = 3}^p \PP\left[R_{n, j} \geq \frac{\epsilon}{2p}t\right] \\
      &\leq \PP\left[X_n \geq \frac{1 - \frac{\epsilon}{2}}{1 - \epsilon} 2t\right] + \sum_{j = 3}^p \PP\left[X_n \geq \frac{1}{C} n^{1 - \frac{2}{j}} t^{\frac{2}{j}}\right]
        \intertext{and provided that we take $C$ even larger and restrict to $t \leq n / C$, we may ensure that all of the terms in the latter sum are at most the first probability, whereby}
      &\leq p \cdot \PP\left[X_n \geq \frac{1 - \frac{\epsilon}{2}}{1 - \epsilon} 2t\right] \,\,\, \text{ if } t \leq \frac{n}{C},
        \intertext{where, defining $\epsilon^{\prime} > 0$ appropriately depending on $\epsilon$, this is equivalently}
      &\leq p \cdot \PP\left[X_n \geq (1 + \epsilon^{\prime}) 2t\right] \,\,\, \text{ if } t \leq \frac{n}{C},
        \intertext{and finally, using the tail bound of Lemma~\ref{lem:pear-right-tail}, for $\delta \in (0, 1)$ to be chosen momentarily, we have}
      &\leq p \left(1 + \frac{2}{\delta}\right)^d \exp\left(-\frac{(1 - \delta)^2(1 + \epsilon^{\prime})}{1 + \frac{1 - \delta}{3} \sqrt{\frac{k^2 - 1}{n}} \sqrt{t}}t\right) \,\,\, \text{ if } t \leq \frac{n}{C}
        \intertext{where choosing $\delta$ small enough and $C$ large enough again, we have}
      &\leq C \exp\left(-\left(1 + \frac{\epsilon^{\prime}}{2}\right)t\right) \,\,\, \text{ if } t \leq \frac{n}{C}\,.
    \end{align*}

    Thus, Lemma~\ref{lem:overlap} applies with the choices
    \begin{align*}
      A(n) &\colonequals \frac{n}{C}, \\
      f(t) &\colonequals C \exp\left(-\frac{\epsilon^{\prime}}{2}t\right).
    \end{align*}
    It remains to check that the condition of Lemma~\ref{lem:overlap} relating $A(n)$ and $D(n)$ holds.
    Recall that this requires that
    \[ A(n) \geq D(n) \left(2 \vee \log\left(\frac{\|R_n\|_{\infty}}{D(n)}\right)\right), \]
    while the assumption of the Theorem guarantees us that
    \[ D(n) \leq C \frac{n}{\log n}. \]
    We must control $\|R_n\|_{\infty}$; to that end, recall that by Proposition~\ref{prop:chi2-pear-infty}, for $X_n \sim \chi^2_{\Pear}(n, k^2)$, we have $\|X_n\|_{\infty} \leq k^2 n$.
    Accordingly, taking $C$ suitably large, by the estimates in \eqref{eq:Rn2-estimate} and \eqref{eq:Rn3-estimate}, we have
    \[ \|R_n\|_{\infty} \leq Cn. \]
    Thus, for sufficiently large $n$,
    \[ D(n) \left(2 \vee \log\left(\frac{\|R_n\|_{\infty}}{D(n)}\right)\right) \leq C n \frac{\log \log n}{\log n} \leq A(n). \]
    Thus the final condition of Lemma~\ref{lem:overlap} is verified, the Lemma applies, and its result gives
    \[ \CAdv_{\leq D(n)}(\QQ_n, \PP_n) \leq \EE \exp^{\leq D(n)}(R_n) = O(1) \]
    as $n \to \infty$, completing the proof.
\end{proof}

\section{Applications}

\subsection{Characteristic Tensors for Discrete Observations}

We first make a few observations about our results over GSBMs where $\Omega$ is a finite set, which is the setting all of our applications will occur in.
\begin{definition}
    We call a GSBM \emph{discrete} if $\Omega$ is a finite set, in which case we identify $\Omega = [\ell]$.
    We call a discrete GSBM \emph{non-degenerate} if, for all $y \in [\ell]$, there exists some $\ba \in [k]^p$ such that $\mu_{\ba}(y) > 0$.
    An equivalent condition is that $\mu_{\avg}(y) > 0$ for all $y \in [\ell]$.
\end{definition}
\noindent
In words, non-degeneracy just asks that there is no $y \in [\ell]$ that occurs in observations $\bY \sim \QQ$ or $\bY \sim \PP$ with probability zero and thus can effectively be removed from the model without changing it.
This is implicitly achieved by the definition of the Radon-Nikodym derivative in our previous formulation, but for more explicit formulas below it will be convenient to make this assumption.

\begin{proposition}
    In a non-degenerate discrete GSBM, the characteristic tensor has entries
    \[ T_{(a_1, b_1), \dots, (a_p, b_p)} = \sum_{y \in [\ell]} \frac{1}{\mu_{\avg}(y)} (\mu_{(a_1, \dots, a_p)}(y) - \mu_{\avg}(y)) (\mu_{(b_1, \dots, b_p)}(y) - \mu_{\avg}(y)). \]
    Said differently, if we view $\bM(y) \in (\RR^{k})^{\otimes p}$ for each $y \in \ell$ as the tensor with $M(y)_{a_1, \dots, a_p} = \mu_{(a_1, \dots, a_p)}(y) - \mu_{\avg}(y)$, then
    \begin{equation}
        \bT = \frac{1}{p!}\sum_{y \in [\ell]} \frac{1}{\mu_{\avg}(y)} \bM(y)^{\otimes 2},
        \label{eq:char-tensor-2}
    \end{equation}
    with the caveat that we ``flatten'' pairs of dimensions in the tensor power $\bM(y)^{\otimes 2}$ (as in the Kronecker product of matrices).
\end{proposition}

Because of the above formulas, it will also be convenient to introduce a notation for the centered channel measures:
\[ \bar{\mu}_{\ba}(y) \colonequals \mu_{\ba}(y) - \mu_{\avg}(y). \]

\subsection{Graph Stochastic Block Models}

\subsubsection{Warmup: Symmetric Two Community Stochastic Block Model}
\label{sec:sbm-warmup}

As a warmup, consider the stochastic block model with two communities and interaction matrix
\begin{equation}
    \bQ = \left[\begin{array}{cc} \alpha & \beta \\ \beta & \alpha \end{array}\right]
\end{equation}
for some constants $\alpha, \beta \geq 0$.

Writing this in our GSBM form, we have $p = k = \ell = 2$ (pairwise observations, of membership in one of two communities, taking binary values), and, writing the probability measures $\mu_{\ba}$ as vectors, we have:
\begin{align*}
  \mu_{(1, 1)} = \mu_{(2, 2)} &= \left(\frac{\alpha}{n}, 1 - \frac{\alpha}{n}\right), \\
  \mu_{(1, 2)} = \mu_{(2, 1)} &= \left(\frac{\beta}{n}, 1 - \frac{\beta}{n}\right), \\
  \mu_{\avg} &= \left(\frac{\alpha + \beta}{2n}, 1 - \frac{\alpha + \beta}{2n}\right), \\
  \bmu_{(1, 1)} = \bmu_{(2, 2)} &= \left(\frac{\alpha - \beta}{2n}, -\frac{\alpha - \beta}{2n}\right), \\
  \bmu_{(1, 2)} = \bmu_{(2, 1)} &= \left(-\frac{\alpha - \beta}{2n}, \frac{\alpha - \beta}{2n}\right).
\end{align*}
Thus the characteristic matrix is:
\begin{align*}
  \bT
  &= \frac{1}{2}\left(\frac{1}{\frac{\alpha + \beta}{2n}} \left[\begin{array}{rr} \frac{\alpha - \beta}{2n} & -\frac{\alpha - \beta}{2n} \\ -\frac{\alpha - \beta}{2n} & \frac{\alpha - \beta}{2n}\end{array}\right]^{\otimes 2} + \frac{1}{1 - \frac{\alpha + \beta}{2n}} \left[\begin{array}{rr} -\frac{\alpha - \beta}{2n} & \frac{\alpha - \beta}{2n} \\ \frac{\alpha - \beta}{2n} & -\frac{\alpha - \beta}{2n}\end{array}\right]^{\otimes 2}\right) \\
  &= \frac{1}{n} \cdot \left(\frac{(\alpha - \beta)^2}{4(\alpha + \beta)} + \frac{1}{n - \frac{\alpha + \beta}{2}}\right) \cdot \left[\begin{array}{rr} 1 & -1 \\ -1 & 1\end{array}\right]^{\otimes 2},
\end{align*}
and its operator norm is
\[ \|\bT\| = \frac{1}{n} \cdot \frac{(\alpha - \beta)^2}{\alpha + \beta} + O\left(\frac{1}{n^2}\right). \]
Theorem~\ref{thm:lcdf-p2} gives that functions of coordinate degree $O(n\, / \log n)$ cannot achieve strong separation in this stochastic block model once the above quantity is smaller by a constant factor than $\frac{k^2}{2n} = \frac{2}{n}$, which coincides with the Kesten-Stigum threshold.
That is, we find that, once
\[ \frac{(\alpha - \beta)^2}{2(\alpha + \beta)} < 1, \]
the above class of functions cannot achieve strong separation.

\subsubsection{Proof of Theorem~\ref{thm:sbm}: General Stochastic Block Model}

We now advance to the general case advertised in Theorem~\ref{thm:sbm}.
The calculations in the proof will be completely analogous to but slightly more abstract than those above.
\begin{proof}[Proof of Theorem~\ref{thm:sbm}]
Suppose that, as in the setting of the Theorem, we have a general stochastic block model, for which the interaction matrix is $\bQ \in [0, 1]^{k \times k}_{\sym}$.
Now we still have $p = \ell = 2$ (pairwise binary observations) but potentially with larger $k$ (number of communities).

Again translating to our language of GSBMs, this corresponds to measures $\mu_{a,b}$ for $a, b \in [k]$ with probability masses
\begin{align*}
  \mu_{a, b} &= \left(\frac{1}{n}Q_{a,b}, 1 - \frac{1}{n}Q_{a,b}\right), \\
  \mu_{\avg} &= \left(\frac{1}{n} \frac{1}{k^2} \one^{\top}\bQ \one, 1 - \frac{1}{n} \frac{1}{k^2} \one^{\top}\bQ \one\right), \\
  \bar{\mu}_{a, b} &= \frac{1}{n} \left(Q_{a, b} - \frac{1}{k^2} \one^{\top}\bQ \one\right) \cdot [\,1,\, -1\,].
\end{align*}

Recall that we have
\[ \bQ\one = \lambda \one \]
for some $\lambda \geq 0$ (which amounts to asking that a vertex in any community has the same average degree), which implies that our model has marginal order at least 2.
In this case, taking the inner product with $\one$ on either side, we see that we must have
\[ \lambda = \frac{1}{k} \one^{\top}\bQ\one, \]
and in terms of this we can rewrite the above as
\begin{align*}
  \mu_{\avg} &= \left(\frac{1}{n} \frac{\lambda}{k}, 1 - \frac{1}{n} \frac{\lambda}{k}\right), \\
  \bar{\mu}_{a, b} &= \frac{1}{n} \left(Q_{a, b} - \frac{\lambda}{k}\right) \cdot [\,1,\, -1\,].
\end{align*}

By the same calculations as before, the characteristic matrix is then
\[ \bT = \frac{1}{2}\left(\frac{1}{\frac{1}{n} \frac{\lambda}{k}} + \frac{1}{1 - \frac{1}{n} \frac{\lambda}{k}}\right) \left(\frac{1}{n} \left(\bQ - \frac{\lambda}{k}\one\one^{\top}\right)\right)^{\otimes 2}, \]
where we note that $\bQ - \frac{\lambda}{k}\one\one^{\top}$ is just $\bQ$ with its unique largest eigenvalue (by the Perron-Frobenius theorem) removed.
Thus the operator norm is
\[ \|\bT\| = \frac{1}{n} \frac{k}{2\lambda_1(\bQ)} \max_{j \in \{2, \dots, k\}} |\lambda_j(\bQ)|^2 + O\left(\frac{1}{n^2}\right). \]
As before, Theorem~\ref{thm:lcdf-p2} gives that functions of coordinate degree $O(n\, / \log n)$ cannot achieve strong separation in this stochastic block model once the above quantity is smaller by a constant factor than $\frac{k^2}{2n}$, which coincides with the (now generalized) Kesten-Stigum threshold, i.e., once
\[ \max_{j \in \{2, \dots, k\}} |\lambda_j(\bQ)|^2 < k\lambda_1(\bQ), \]
then the above class of functions cannot achieve strong separation, as claimed.
\end{proof}

\subsection{Hypergraph Stochastic Block Model}

\subsubsection{Warmup: Symmetric Two Community Hypergraph Stochastic Block Model}
\label{sec:hsbm-warmup}

As before, let us begin with the simpler case of two symmetric communities.
Let us be explicit in this manageable setting and write the values of the various tensors involved directly.
To that end, let $\bD \in (\{0, 1\}^2)^{\otimes p}$ be the \emph{diagonal} tensor whose entries are $D_{a, \dots, a} = 1$ for $a \in \{1, 2\}$ and all other entries zero, and let $\bF \in (\{0, 1\}^2)^{\otimes p}$ be the \emph{off-diagonal} tensor, $\bF = [1, 1]^{\otimes p} - \bD$, having the opposite pattern of entries.
Very concretely, we may express
\begin{align*}
  \bD &= \be_1^{\otimes p} + \be_2^{\otimes p}, \\
  \bF &= \one^{\otimes p} - \bD \\
      &= \one^{\otimes p} - \be_1^{\otimes p} - \be_2^{\otimes p}.
\end{align*}

Then, the interaction tensor of the symmetric two community HSBM is
\begin{align*}
  \bQ = \alpha \bD + \beta \bF \in ([0, 1]^2)^{\otimes p}.
\end{align*}
For the GSBM parameters, we still have $k = \ell = 2$ (two communities and binary values of the observations), but now have general $p$ ($p$-ary observations) unlike the graph SBM.

Writing the rest of the setup in the GSBM language, we have:
\begin{align*}
  \mu_{(1, \dots, 1)} = \mu_{(2, \dots, 2)} &= \left(\frac{a}{\binom{n}{p - 1}}, 1 - \frac{a}{\binom{n}{p - 1}}\right), \\
  \mu_{\bm a} &= \left(\frac{b}{\binom{n}{p - 1}}, 1 - \frac{b}{\binom{n}{p - 1}}\right) \text{ for all } \bm a \notin\{(1, \dots, 1), (2, \dots, 2)\}, \\
  \mu_{\avg} &= \left(\frac{a + (2^{p - 1} - 1)b}{2^{p - 1}\binom{n}{p - 1}}, 1 - \frac{a + (2^{p - 1} - 1)b}{2^{p - 1}\binom{n}{p - 1}}\right), \\
  \bmu_{(1, \dots, 1)} = \bmu_{(2, \dots, 2)} &= \left(\frac{2^{p - 1} - 1}{2^{p - 1}\binom{n}{p - 1}}(a - b), -\frac{2^{p - 1} - 1}{2^{p - 1}\binom{n}{p - 1}}(a - b)\right), \\
  \bmu_{\bm a} &= \left(-\frac{1}{2^{p - 1}\binom{n}{p - 1}}(a - b), \frac{1}{2^{p - 1}\binom{n}{p - 1}}(a - b)\right) \text{ for all } \bm a \notin\{(1, \dots, 1), (2, \dots, 2)\}.
\end{align*}
The characteristic tensor is therefore
\[ \bT = \bT^{(p)} = \frac{1}{p!}\left(\frac{2^{p - 1}\binom{n}{p - 1}}{a + (2^{p - 1} - 1)b}  + \frac{1}{1 - \frac{a + (2^{p - 1} - 1)b}{2^{p - 1}\binom{n}{p - 1}}}\right) \left(\frac{a - b}{2^{p - 1}\binom{n}{p - 1}}\right)^2 \bM^{\otimes 2}, \]
where, using our previous expressions for $\bD$ and $\bF$,
\begin{align*}
  \bM
  &= (2^{p - 1} - 1)\bD - \bF \\
  &= (2^{p - 1} - 1)(\be_1^{\otimes p} + \be_2^{\otimes p}) - (\one^{\otimes p} - \be_1^{\otimes p} - \be_2^{\otimes p}) \\
  &= 2^{p - 1}(\be_1^{\otimes p} + \be_2^{\otimes p}) - \one^{\otimes p}.
\end{align*}
Note that, when $p = 2$, this is compatible with our result from the corresponding first calculation of Section~\ref{sec:sbm-warmup}.

We may now observe that the marginal order of this model is indeed 2: we have $\bM[\one, \dots, \one, \cdot\,] = \bm 0$ and therefore $\bT^{(1)} = \bm 0$, while the next-smallest marginalization is
\begin{align*}
  \bT^{(2)}
  &= \frac{1}{p!\cdot 2^{2(p - 2)}}\cdot \left(\frac{2^{p - 1}\binom{n}{p - 1}}{a + (2^{p - 1} - 1)b}  + \frac{1}{1 - \frac{a + (2^{p - 1} - 1)b}{2^{p - 1}\binom{n}{p - 1}}}\right) \left(\frac{a - b}{2^{p - 1}\binom{n}{p - 1}}\right)^2 \\
  &\hspace{2.25cm} \cdot \,\, 2^{2p - 2}\left((\be_1^{\otimes 2} + \be_2^{\otimes 2}) - \frac{1}{2}\one^{\otimes 2}\right)^{\otimes 2}
    \intertext{which, canceling a factor and rewriting as a matrix, is}
  &= \frac{4}{p!}\left(\frac{2^{p - 1}\binom{n}{p - 1}}{a + (2^{p - 1} - 1)b}  + \frac{1}{1 - \frac{a + (2^{p - 1} - 1)b}{2^{p - 1}\binom{n}{p - 1}}}\right) \left(\frac{a - b}{2^{p - 1}\binom{n}{p - 1}}\right)^2 \cdot \left(\bm I_2 - \frac{1}{2} \one_2\one_2^{\top}\right)^{\otimes 2}.
\end{align*}
Thus the bottom line is much the same as before: this remaining matrix is just the orthogonal projection matrix to the direction orthogonal to $\one_2$, and thus the operator norm of this marginalized characteristic matrix is
\begin{align*}
  \|\bT^{(2)}\|
  &= \frac{4}{p!}\left(\frac{2^{p - 1}\binom{n}{p - 1}}{a + (2^{p - 1} - 1)b}  + \frac{1}{1 - \frac{a + (2^{p - 1} - 1)b}{2^{p - 1}\binom{n}{p - 1}}}\right) \left(\frac{a - b}{2^{p - 1}\binom{n}{p - 1}}\right)^2 \\
  &= \frac{4}{p!}\cdot \frac{1}{\binom{n}{p - 1}} \cdot \frac{(a - b)^2}{2^{p - 1}(a + (2^{p - 1} - 1)b)} + O\left(\frac{1}{n^p}\right) \\
  &= \frac{1}{n^{p - 1}} \cdot \frac{4(a - b)^2}{p2^{p - 1}(a + (2^{p - 1} - 1)b)} + O\left(\frac{1}{n^p}\right).
\end{align*}
Theorem~\ref{thm:lcdf-p2} gives that functions of coordinate degree $O(n\, / \log n)$ cannot achieve strong separation in this HSBM once the above quantity is smaller by a constant factor than $\frac{k^2}{p(p - 1)n^{p - 1}} = \frac{4}{p(p - 1)n^{p - 1}}$.
In particular, it suffices to have
\[ \frac{(p - 1)(a - b)^2}{2^{p - 1}(a + (2^{p - 1} - 1)b)} < 1, \]
which is precisely the HSBM version of the Kesten-Stigum threshold.

\subsubsection{Proof of Theorem~\ref{thm:hsbm}: General Hypergraph Stochastic Block Model}

\begin{proof}[Proof of Theorem~\ref{thm:sbm}]
Suppose that, as in the setting of the Theorem, we have a general stochastic block model with a fixed symmetric interaction tensor $\bQ \in ([0, 1]^{k})^{\otimes p}$.
We are now in a setting where $\ell = 2$ (binary observations) but $p$ (arity of interactions) and $k$ (number of communities) are arbitrary.

Again translating to our language of GSBMs, this corresponds to measures $\mu_{\ba}$ for $\ba \in [k]^p$ with probability masses
\begin{align*}
  \mu_{\ba} &= \left[\,\frac{1}{\binom{n}{p - 1}}Q_{\ba},\, 1 - \frac{1}{\binom{n}{p - 1}}Q_{\ba}\,\right], \\
  \mu_{\avg} &= \left[\,\frac{1}{\binom{n}{p - 1}}\frac{1}{k^p} \langle \bQ, \one^{\otimes p} \rangle,\, 1 - \frac{1}{\binom{n}{p - 1}}\frac{1}{k^p} \langle \bQ, \one^{\otimes p} \rangle\,\right], \\
  \bar{\mu}_{\ba} &= \frac{1}{\binom{n}{p - 1}}\left(Q_{\ba} - \frac{1}{k^p} \langle \bQ, \one^{\otimes p} \rangle\right) \cdot \big[\,1,\,-1\,\big].
\end{align*}

Recall that we have assumed that
\begin{equation}
    \label{eq:1-tensor-evec}
    \bQ[\one, \dots, \one, \cdot\,] = \lambda\one.
\end{equation}
We again have
\[ \lambda = \frac{1}{k} \langle \bQ, \one^{\otimes p} \rangle, \]
whereby we can rewrite
\begin{align*}
  \mu_{\avg} &= \left(\frac{1}{\binom{n}{p - 1}} \frac{\lambda}{k^{p - 1}}, 1 - \frac{1}{\binom{n}{p - 1}} \frac{\lambda}{k^{p - 1}}\right), \\
  \bar{\mu}_{\ba} &= \frac{1}{\binom{n}{p - 1}} \left(Q_{\ba} - \frac{\lambda}{k^{p - 1}}\right) \cdot \big[\,1,\,-1\,\big]
\end{align*}

The characteristic tensor is then
\[ \bT^{(p)} = \frac{1}{p!}\left(\frac{1}{\frac{1}{\binom{n}{p - 1}} \frac{\lambda}{k^{p - 1}}} + \frac{1}{1 - \frac{1}{\binom{n}{p - 1}} \frac{\lambda}{k^{p - 1}}}\right)\frac{1}{\binom{n}{p - 1}^2}\left(\bQ - \frac{\lambda}{k^{p - 1}}\one^{\otimes p}\right)^{\otimes 2}. \]
We see that this model will indeed have marginal order at least 2, as the marginalized characteristic tensor $\bT^{(1)}$ (a 1-tensor, or vector) will be zero by our assumption \eqref{eq:1-tensor-evec}, while the next marginalization will be
\[ \bT^{(2)} = \frac{1}{p! \cdot k^{2(p - 2)}}\left(\frac{1}{\frac{1}{\binom{n}{p - 1}} \frac{\lambda}{k^{p - 1}}} + \frac{1}{1 - \frac{1}{\binom{n}{p - 1}} \frac{\lambda}{k^{p - 1}}}\right)\frac{1}{\binom{n}{p - 1}^2}\left(\bQ[\one, \dots, \one, \cdot, \cdot] - \frac{\lambda}{k}\one^{\otimes 2}\right)^{\otimes 2}. \]

Note now that $\bQ[\one, \dots, \one, \cdot, \cdot]$ is a non-negative symmetric matrix, and by our assumption \eqref{eq:1-tensor-evec}, its Perron-Frobenius eigenvalue is $\lambda$ with eigenvector $\one$.
Thus the matrix appearing above (as a 2-tensor) is just the matrix $\bQ[\one, \dots, \one, \cdot, \cdot]$ with its top eigenvector removed.
The operator norm of $\bT^{(2)}$ is therefore
\begin{align*}
  \|\bT^{(2)}\|
  &= \frac{1}{p! \cdot k^{p - 3} \cdot \binom{n}{p - 1}\cdot \lambda} \max_{j = 2, \dots, k} |\lambda_j(\bQ[\one, \dots, \one, \cdot, \cdot])|^2 + O\left(\frac{1}{n^p}\right) \\
  &= \frac{1}{n^{p - 1}} \cdot \frac{1}{p \cdot k^{p - 3}\cdot \lambda} \max_{j = 2, \dots, k} |\lambda_j(\bQ[\one, \dots, \one, \cdot, \cdot])|^2 + O\left(\frac{1}{n^p}\right)
\end{align*}
As before, Theorem~\ref{thm:lcdf-p2} gives that functions of coordinate degree $O(n\, / \log n)$ cannot achieve strong separation in this stochastic block model once the above quantity is smaller by a constant factor than $\frac{k^2}{p(p - 1)n}$, which coincides with the (now generalized) Kesten-Stigum threshold, i.e., once
\[ \max_{j \in \{2, \dots, k\}} |\lambda_j(\bQ[\one, \dots, \one, \cdot, \cdot])|^2 < \frac{k^{p - 1}}{p - 1} \lambda_1(\bQ[\one, \dots, \one, \cdot, \cdot]), \]
then the above class of functions cannot achieve strong separation.
\end{proof}

\subsubsection{Higher Marginal Order Example: Random XOR-SAT}
\label{sec:xor}

As an example of an interesting model with marginal order higher than 2, let us describe how a lower bound matching several prior works on the random $p$-XOR-SAT problem can be seen in our framework.
Our first task is to embed a version of random $p$-XOR-SAT into a GSBM.
Recall that a $p$-XOR-SAT instance is equivalently a system of parity equations of the form $x_{i_1} \cdots x_{i_p} = b$ with $b \in \{\pm 1\}$, to be solved over $\bx \in \{\pm 1\}^n$.
Each such equation is sometimes called a \emph{clause}, and we will follow standard notation in writing $m$ for the number of these clauses.
We consider hypothesis testing problems that ask to distinguish between uniformly random such instances (in a sense we will clarify in a moment) and ones with a planted structure causing unusually many clauses to be satisfiable at once.

Consider a GSBM of order $p$ with $\Omega = \{+1, -1, \bullet\}$, where $\bullet$ stands for a clause that is not included in an instance.
We set $k = 2$, but identify the labels not with $[2]$ but again with $\{\pm 1\}$.
We then define, for $\eta \in (0, 1)$,
\[ \mu_{(a_1, \dots, a_p)} \colonequals \eta \delta_{a_1 \cdots a_p} + (1 - \eta) \delta_{\bullet}, \]
which will yield the average
\[ \mu_{\avg} = \eta \Unif(\{\pm 1\}) + (1 - \eta) \delta_{\bullet}. \]
A sample from the null model $\QQ$ of this GSBM will then be, in effect, a $p$-XOR-SAT instance where every clause is present with probability $\eta$ independently and the right-hand sides are drawn uniformly at random from $\{\pm 1\}$.
A sample from the planted model $\PP$ will be an instance where every clause is again present with probability $\eta$ independently, but there is a uniformly random satisfying assignment $\bx$ chosen that determines the right-hand sides so that it is satisfying.

Our discussion of lower bounds will actually apply to this exactly satisfiable model; as has been widely discussed, the noiseless p-XOR-SAT problem is solvable by Gaussian elimination, but many frameworks for understanding algorithmic hardness more generally only address robust algorithms and neglect this ``brittle'' algebraic algorithm.
See, e.g., discussion in \cite{KWB-2022-LowDegreeNotes}).
One may also add a small amount of noise by resampling (in the sense of Definition~\ref{def:resampling}) with a small probability.
In fact, our construction above is just the $(1 - \eta)$-censorship (in the sense of Definition~\ref{def:censorship}) of the GSBM that reveals all clauses.
By our discussion in Section~\ref{sec:channel-calc}, rates of censorship and resampling have the same effect on our low-degree calculations, so a small amount of further resampling noise will effectively just change $\eta$ slightly, and we may just as well repeat all calculations after that operation.
We continue calculating with the exactly satisfiable model above for the purposes of this discussion.

We compute the centered likelihood ratios
\[ \frac{d\mu_{(a_1, \dots, a_p)}}{d\mu_{\avg}}(y) - 1 = \left\{\begin{array}{ll} 0 & \text{if } y = \bullet, \\ (-1)^{\One\{a_1 \cdots a_p = y\}} & \text{if } y \in \{\pm 1\} \end{array}\right\}, \]
from which we compute the entries of the characteristic tensor, which end up with a simple form:
\begin{align*}
  T_{(a_1, b_1), \dots, (a_p, b_p)}
  &= \Ex_{y \sim \mu_{\avg}}\left[\left(\frac{d\mu_{(a_1, \dots, a_p)}}{d\mu_{\avg}}(y) - 1 \right)\left(\frac{d\mu_{(b_1, \dots, b_p)}}{d\mu_{\avg}}(y) - 1 \right)\right] \\
  &= \eta \Ex_{y \sim \Unif(\{\pm 1\})} (-1)^{\One\{a_1 \cdots a_p = y\} + \One\{b_1 \cdots b_p = y\}} \\
  &= \eta \cdot (-1)^{\One\{a_1 \cdots a_p = b_1 \cdots b_p\}} \\
  &= \eta \cdot a_1b_1 \cdots a_pb_p,
\end{align*}
where we use that we have chosen to work with $a_i, b_j \in \{\pm 1\}$.
Note that, up to an appropriate permutation of all axes, this means that
\[ \bT = \eta \left[\begin{array}{r} +1 \\ -1 \\ -1 \\ +1 \end{array}\right]^{\otimes p}. \]
In particular then, every marginalization $\bT^{(j)}$ for $j < p$ will be zero (since the vector above is orthogonal to $\one$), and the marginal order of this model is
\[ p_* = p. \]

Thus, we may apply Theorem~\ref{thm:lcdf-p3} and find that, given a sequence of degrees $D(n)$, no LCDF of degree $D(n)$ will achieve strong separation provided that
\[ \eta = \eta(n) = O\left(n^{-\frac{p}{2}}D(n)^{-\frac{p - 2}{2}}\right). \]
Since the number of clauses (in expectation and typically within small error) in such random formulas is about $m = \eta \cdot \binom{n}{p} \sim \eta \cdot n^p$, this condition is equivalent to having the number of clauses scale as
\[ m = m(n) = O\left(n^{\frac{p}{2}}D(n)^{-\frac{p - 2}{2}}\right). \]
In the regime $D(n) = n^{\delta}$, this lower bound precisely complements (under the heuristic that LCDF of degree $D(n)$ correspond to $\exp(\widetilde{\Theta}(D(n)))$-time algorithms) the scaling of the runtime of at least two families of subexponential time algorithms for this problem (or its variant with a small amount of noise): the spectral and sum-of-squares algorithms of \cite{RRS-2017-StronglyRefutingCSP}, and the simpler spectral algorithms based on the Kikuchi hierarchy of \cite{WEAM-2019-KikuchiTensorPCA}.

\subsection{Group Problems}

Recall that for the problems over finite groups $G$ that we study, we write $k \colonequals |G|$, which coincides with the parameter $k$ of the GSBMs we consider.

\subsubsection{Proof of Theorem~\ref{thm:sync}: Truth-or-Haar Synchronization}

In the truth-or-Haar model of group synchronization, we choose $n$ elements of $G$ independently and uniformly at random, say $g_1, \dots, g_n$, and for each $i < j$ we observe $y_{ij}$ which is $g_i g_j^{-1}$ with probability $\eta = \eta(n)$ and a uniformly random element of $G$ with probability $1 - \eta$ (hence the name, since the Haar measure on $G$ is just the uniform measure).
We will ultimately be interested in the scaling $\eta(n) = \gamma / \sqrt{n}$ for a constant $\gamma$, but for convenience for now work with the $\eta$ parameter instead.
Let us write $\mu_{\avg}$ for the uniform measure on $G$ (which will indeed be $\mu_{\avg}$ in the GSBM setting as well, as we will see momentarily).

\begin{proof}[Proof of Theorem~\ref{thm:sync}]
Writing this as a GSBM, the measures involved are indexed by pairs $g, h \in G$, and we have
\[ \mu_{(g, h)} = \eta \delta_{g h^{-1}} + (1 - \eta) \mu_{\avg} \]
by the above definition.
We see that $\mu_{\avg}$ is indeed the average of the $\mu_{(g, h)}$.
We note also that this model is \emph{not} strongly symmetric for most $G$, since $gh^{-1} \neq hg^{-1} = (gh^{-1})^{-1}$ unless every element of $G$ has order 1 or 2.
However, from the calculations below we will see that the characteristic matrix is symmetric and thus that the model is weakly symmetric, so our tools still apply.

\begin{remark}
    We see also that in fact this model can be obtained as the $(1 - \eta)$-resampling (in the sense of Definition~\ref{def:resampling}) of the ``noiseless'' GSBM where we start out with simply observing the exact group element differences, $\mu_{(g, h)} = \delta_{gh^{-1}}$.
\end{remark}

We then also have
\[ \bar{\mu}_{(g, h)} = \eta(\delta_{g h^{-1}} - \mu_{\avg}). \]
Let $\bP_g \in \{0, 1\}^{k \times k}$ be the permutation matrix associated to the permutation of $g$ acting by multiplication on the left on $G$ (i.e., the image of $G$ under the left regular representation of $G$).
The characteristic matrix is then
\begin{align*}
  \bT
  &= \frac{\eta^2}{2} \sum_{g \in G}k \left(\bP_g - \frac{1}{k}\one_k\one_k^{\top}\right)^{\otimes 2} \\
  &= \frac{\eta^2}{2}\Bigg(k \underbrace{\sum_{g \in G} \bP_g^{\otimes 2}}_{\equalscolon \bM} - \one_{k^2}\one_{k^2}^{\top},\Bigg)
\end{align*}
where we expand the tensor product and use that $\sum_{g \in G} \bP_g = \one_k\one_k^{\top}$.
From this expression it follows that $\bT$ is a symmetric matrix and thus that the GSBM is weakly symmetric, since $(\bP_g)^{\top} = \bP_{g^{-1}}$ so transposition merely permutes the terms in the summation above.

We must understand the eigenvalues of this matrix.
Consider the entries of the matrix $\bM$: it is indexed by pairs $(g, h) \in G^2$, and we have
\begin{align*}
  (\bP_f^{\otimes 2})_{(g_1, h_1), (g_2, h_2)} &= \One\{ g_1g_2^{-1} = h_1h_2^{-1} = f\}, \\
  \bM_{(g_1, h_1), (g_2, h_2)} &= \One\{ g_1g_2^{-1} = h_1h_2^{-1}\} \\ &= \One\{ g_1^{-1}h_1 = g_2^{-1}h_2 \}.
\end{align*}
But, this latter expression means that $\bM$, after suitably permuting the rows and columns, is just the block matrix $\bm I_k \otimes \one_k\one_k^{\top}$, whose diagonal $k \times k$ blocks are the all-ones matrix and whose other blocks are zero.
$\one_{k^2}$ is an eigenvector of this matrix with eigenvalue $k$, but this eigenspace has dimension $k > 1$.
Therefore, we have
\[ \|\bT\| = \frac{k^2 \eta^2}{2}. \]

Theorem~\ref{thm:lcdf-p2} gives that functions of coordinate degree $O(n\, / \log n)$ cannot achieve strong separation in this GSBM once the above quantity is smaller by a constant factor than $\frac{k^2}{2n}$.
Thus, if we have
\[ \eta = \eta(n) < \frac{1 - \epsilon}{\sqrt{n}}, \]
i.e. if $\gamma < 1$, then we obtain the stated lower bound.
\end{proof}

\subsubsection{Proof of Theorem~\ref{thm:sumset}: Truth-or-Haar Sumset}

\begin{proof}[Proof of Theorem~\ref{thm:sumset}]
This model is identical to the synchronization model, but instead of observing $gh^{-1}$ or a uniformly random element, we observe $gh$ or a uniformly random element.
In other words, our GSBM is now specified by channel measures
\[ \mu_{(g, h)} = \eta \delta_{g h} + (1 - \eta) \mu_{\avg} \]
Recall that we further assume in this case that $G$ is abelian.

At a glance, this model might seem quite different from the previous one.
Indeed, the intermediate calculations are different: we may carry out the same plan, but instead of working with the permutation matrices $\bP_f$ with entries
\[ (\bP_f)_{gh} = \One\{gh^{-1} = f\}, \]
we must work with $\widetilde{\bP}_f$ with entries
\[ (\widetilde{\bP}_f)_{gh} = \One\{gh = f\}. \]
These are still permutations and still have $\sum_{g \in G} \widetilde{\bP}_g = \one_k\one_k^{\top}$, so the first part of the argument remains unaffected, and we reach the characteristic matrix
\[ \bT = \frac{\eta^2}{2}\Bigg(k \underbrace{\sum_{g \in G} \widetilde{\bP}_g^{\otimes 2}}_{\equalscolon \widetilde{\bM}} - \one_{k^2}\one_{k^2}^{\top},\Bigg). \]

In understanding $\widetilde{\bM}$, we must use that now we assume $G$ is abelian: without that assumption, we could only reach
\begin{align*}
  (\widetilde{\bP}_f^{\otimes 2})_{(g_1, h_1), (g_2, h_2)}
  &= \One\{ g_1g_2 = h_1h_2 = f\}, \\
  \widetilde{\bM}_{(g_1, h_1), (g_2, h_2)}
  &= \One\{ g_1g_2 = h_1h_2\} \\
  &= \One\{ g_1^{-1}h_1 = g_2h_2^{-1} \}
    \intertext{but assuming that $G$ is abelian we may continue}
  &= \One\{ g_1^{-1}h_1 = h_2^{-1}g_2 \} \\
  &= \One\{ g_1^{-1}h_1 = (g_2^{-1}h_2)^{-1} \}.
\end{align*}
Thus $\widetilde{\bM}$ is the same as $\bM$ from the group synchronization case, but with an involution applied to the columns.
After permuting the rows and columns such that this involution reverses the order of the columns, $\widetilde{\bM}$ is then $\widetilde{\bm I}_k \otimes \one_k\one_k^{\top}$, where $\widetilde{\bm I}_k$ is the $k \times k$ ``anti-identity'' matrix:
\[ \widetilde{\bm I}_k = \left[\begin{array}{ccccc} & & & & 1 \\ & & & 1 & \\ & & \rotatebox{76}{$\ddots$} & & \\ & 1 & & & \\ 1 & & & & \end{array}\right]. \]
As this matrix is symmetric, so is $\widetilde{\bM}$, and we see that $\bT$ is symmetric also and so this model is weakly symmetric.
But, this would not necessarily be the case if $G$ were not abelian (like for small non-abelian groups including the symmetric group $\Sym([3])$).

It is easily verified that $\widetilde{\bm I}_k$ has the eigenvalue 1 with multiplicity $\lceil k / 2 \rceil$, and the remaining eigenvalues are all $-1$.
Thus, we still always have
\[ \|\bT\| = \frac{k^2\eta^2}{2}, \]
and the remainder of the argument goes through as for synchronization to derive the (identical) lower bound against LCDF.
\end{proof}

\subsubsection{Gaussian Multi-Frequency Synchronization}
\label{sec:kbk}

We briefly summarize how our technical results in Section~\ref{sec:pearson} imply an improvement to the recent results of \cite{KBK-2024-LowDegreeGaussianSynchronization} on a related synchronization model.
For the sake of explanation let us focus on cyclic groups, though the improvement also holds for their results on general finite groups and also the circle group $U(1)$ (which is analyzed by reducing to cyclic groups).
That work studies the following model for the cyclic groups.
We write $\mathsf{GUE}(n)$ for the law of a random Hermitian matrix $\bW \in \CC^{n \times n}$ whose entries on and above the diagonal are independent, with diagonal entries distributed as $\sN(0, 2)$ and off-diagonal entries having real and imaginary parts distributed independently as $\sN(0, 1/2)$.
\begin{definition}[Multi-frequency $\ZZ_k$ synchronization]\label{def:angular_synch_ZL}
    Let $k \geq 2$ and $\lambda \ge 0$.
    We consider two models of $\ceil{k/2} - 1$ observations of $n \times n$ Hermitian matrices:
    \begin{enumerate}
    \item Under $\QQ_n$, observe $\bY_{\ell} \sim \mathsf{GUE}(n)$ for $\ell = 1, \dots, \ceil{k/2} - 1$.
    \item Under $\PP_n$, draw $\bx \in \CC^n$ a random vector such that each entry is sampled independently as $x_j \sim \Unif(\{\omega^0, \dots, \omega^{k-1}\})$, where $\omega = \exp(2\pi i / k)$ is a primitive root of unity.
        Also, draw $\bW_{\ell} \sim \mathsf{GUE}(n)$ for $\ell = 1, \dots, \ceil{k/2} - 1$.
        Then, observe
        \begin{equation*}
        \bY_\ell = \frac{\lambda}{n} \bx^{(\ell)} \bx^{(\ell)*} + \frac{1}{\sqrt{n}} \bW_\ell, \text{ for } \ell = 1, \ldots \ceil{k/2} - 1,
    \end{equation*}
    where $\bx^{(\ell)}$ denotes the $\ell$th entrywise power.
    \end{enumerate}
\end{definition}

The general conjecture concerning this model, first proposed by \cite{PWBM-2016-PCASpikedMatrixSynchronization}, is that testing should be computationally hard whenever $\lambda < 1$ (the same as the computational threshold for a single such spiked matrix observation, meaning that, perhaps surprisingly, the presence of extra observations is not helpful).
In \cite{KBK-2024-LowDegreeGaussianSynchronization}, it is shown (in our terminology) that no sequence of polynomials of degree $o(n^{1/3})$ can strongly separate $\QQ_n$ from $\PP_n$.
We improve on this scaling (and, via Corollary~\ref{cor:simple-moment} and our work in Section~\ref{sec:pearson}, give a more conceptual argument for the technical details) as follows:

\begin{theorem}
    For any $\lambda < 1$, there exists $\gamma = \gamma(\lambda) > 0$ such that no polynomial of degree at most $\gamma n$ can strongly separate $\QQ_n$ from $\PP_n$ in the model of Definition~\ref{def:angular_synch_ZL}.
\end{theorem}
\begin{proof}
    We will work with the polynomial advantage, denoted $\Adv_{\leq D}$ and given by the same expression as in our Definition~\ref{def:cadv} of the coordinate advantage, but optimizing over $\deg(f) \leq D$ for $f$ a polynomial rather than $\cdeg(f) \leq D$:
    \[ \Adv_{\leq D}(\PP, \QQ) \colonequals \left\{\begin{array}{ll} \text{maximize} & \Ex_{\by \sim \PP} f(\by) \\ \text{subject to} & \Ex_{\by \sim \QQ} f(\by)^2 \leq 1, \\ & f \in \RR[\by], \\ & \deg(f) \leq D \end{array}\right\}. \]
    The analog of Proposition~\ref{prop:cadv-sep} holds for the polynomial advantage and strong separation by polynomials, so that to prove the Theorem it suffices to show that $\Adv_{\leq D(n)}(\PP_n, \QQ_n) = O(1)$ whenever $D(n) \leq \gamma n$.

    Let us choose $\epsilon = \epsilon(\lambda) > 0$ such that $\lambda^2 \cdot \frac{2 + \epsilon}{2} \leq 1 - \epsilon$.
    We compute starting from Equation~(6.4) of \cite{KBK-2024-LowDegreeGaussianSynchronization}, which expresses the advantage in terms of expectations over multinomial vectors.
    In fact, we may relate the expectations appearing there to Pearson's $\chi^2$ statistics:
    \begin{align*}
      \Adv_{\leq D(n)}(\PP_n, \QQ_n)^2
      &= \sum_{d = 0}^{D(n)} \frac{1}{d!} \frac{\lambda^{2d}}{n^d} \Ex_{\bz \sim \Mult(n, k)}\left(\frac{k}{2}\sum_{\ell = 1}^k \left(z_{\ell} - \frac{n}{k}\right)^2\right)^d \\
      &= \sum_{d = 0}^{D(n)} \frac{1}{d!} \frac{\lambda^{2d}}{2^d} \Ex_{X \sim \chi^2_{\Pear}(n, k)}X^d
        \intertext{and now by Corollary~\ref{cor:simple-moment}, there exist $C, \gamma > 0$ depending on our $\epsilon$ chosen above (and thus in turn depending only on $\lambda$) such that, if $D(n) \leq \gamma n$, then}
      &\leq C^k \sum_{d = 0}^{D(n)} \frac{1}{d!} \frac{\lambda^{2d}}{2^d} d^{3/2} \left(\frac{(2 + \epsilon)d}{e}\right)^d \\
      &\leq C^k \sum_{d = 0}^{D(n)} \frac{1}{d!} d^{3/2} \left(\frac{(1 - \epsilon)d}{e}\right)^d
      \intertext{and bounding the factorial by Proposition~\ref{prop:factorial} we have}
      &\leq C^k \sum_{d = 0}^{\infty} d^{3/2} (1 - \epsilon)^d
    \end{align*}
    which converges, completing the proof.
\end{proof}

\section*{Acknowledgments}
\addcontentsline{toc}{section}{Acknowledgments}

Thanks to Afonso Bandeira and Anastasia Kireeva for helpful discussions about group synchronization, to Yuzhou Gu for discussions, references, and suggestions about hypergraph stochastic block models and their Kesten-Stigum thresholds, and to Yuval Ishai for suggesting (a different version of) the noisy group sumset problem.
Part of this work was done while the author was employed by the Department of Computer Science of Yale University.
Inspiration for parts of this work was also drawn from discussions at the Workshop on Computational Complexity of Statistical Inference at the Banff International Research Station in February 2024.

\phantomsection
\bibliographystyle{alpha}
\addcontentsline{toc}{section}{References}
\bibliography{main}

\end{document}